\documentclass[sts]{imsart-mod}

\usepackage{amsthm,amsmath, amssymb, color}

\usepackage[utf8]{inputenc}
\usepackage{amsmath}
\usepackage{amssymb}
\usepackage{bbm, bm}
\usepackage{enumitem}

\usepackage{color}
\usepackage{amsfonts, fancybox}
\usepackage{amssymb}
\usepackage[american]{babel}
\usepackage{graphicx, pstricks, pst-node, color}
\usepackage{psfrag,color}
\usepackage{dcolumn}
\usepackage{subfigure}
\usepackage{flafter, bm, dsfont}
\usepackage[section]{placeins}

\usepackage{multirow}
\usepackage{color}
\usepackage{amsmath}
\usepackage{float}
\usepackage{mathrsfs}
\usepackage{amsfonts}
\usepackage{dsfont}
\usepackage{amssymb}
\usepackage[boxed]{algorithm2e}
\usepackage{amssymb, latexsym}

\usepackage{graphicx}
\usepackage{epstopdf}

\usepackage{boxedminipage}

\newcommand{\bc}{\begin{center}}
\newcommand{\ec}{\end{center}}
\newcommand{\ba}{\begin{array}}
\newcommand{\ea}{\end{array}}
\newcommand{\be}{\begin{eqnarray}}
\newcommand{\ee}{\end{eqnarray}}
\newcommand{\bel}{\begin{eqnarray}\label}
\newcommand{\El}{\end{eqnarray}}
\newcommand{\bes}{\begin{eqnarray*}}
\newcommand{\Es}{\end{eqnarray*}}
\newcommand{\bn}{\begin{enumerate}}
\newcommand{\en}{\end{enumerate}}

\definecolor{MIT}{cmyk}{.24, 1.00, .78, .17} 
\definecolor{pink}{cmyk}{0, 1, 0, 0} 
\definecolor{darkgreen}{cmyk}{1,0, 1, 0} 

\newtheorem{lemma}{Lemma}

\newtheorem{proposition}{Proposition}

\newtheorem{theorem}{Theorem}
\newtheorem{corollary}{Corollary}

\newtheorem*{theorem*}{Theorem}
\newtheorem*{lemma*}{Lemma}
\newtheorem*{proposition*}{Proposition}
\newtheorem*{corollary*}{Corollary}

{
	\begin{center}
		\begin{boxedminipage}{0.8\linewidth}
			\begin{center}
				\textbf{\texttt{#1}}
			\end{center}
			\rm
			\begin{tabbing}
				....\=...\=...\=...\=...\=  \+ \kill
			} %
			{\end{tabbing}
	\end{boxedminipage} \end{center} 
}

\newcommand{\flo}[1]{\lfloor #1 \rfloor} 
\newcommand{\ceil}[1]{\lceil #1 \rceil} 
\newcommand{\eps}{\varepsilon} 
\newcommand{\mb}[1]{\mathbb{#1}} 
\newcommand{\mc}[1]{\mathcal{#1}} 
\newcommand{\abs}[1]{\left|#1\right|} 
\newcommand{\norm}[1]{\lVert#1\rVert} 
\newcommand{\hd}{ \mc{P}_\mc{H} } 
\newcommand{\sd}{ \mc{S} } 
\newcommand{\unif}{\hat f_S^{\mathsf{unif}}} 
\newcommand{\Cara}{Carath\'{e}odory } 
\newcommand{\ti}{\tilde } 
\newcommand{\disc}{\mathsf{disc} } 

\newcommand{\cbeta}{\lfloor \beta \rfloor}

\newcommand{\cC}{\mathcal{C}}
\newcommand{\cD}{\mathcal{D}}
\newcommand{\cE}{\mathcal{E}}

\newcommand{\cG}{\mathcal{G}}

\newcommand{\cN}{\mathcal{N}}

\newcommand{\cP}{\mathcal{P}}

\newcommand{\cV}{\mathcal{V}}
\newcommand{\cW}{\mathcal{V}}

\newcommand{\ud}{\mathrm{d}}

\newcommand{\h}{{\rm I}\kern-0.18em{\rm H}}
\newcommand{\K}{{\rm I}\kern-0.18em{\rm K}}
\newcommand{\p}{\mb{P}}
\newcommand{\E}{\mb{E}}
\newcommand{\Z}{{\rm Z}\kern-0.18em{\rm Z}}
\newcommand{\1}{{\rm 1}\kern-0.24em{\rm I}}
\newcommand{\N}{{\rm I}\kern-0.18em{\rm N}}

\usepackage{hyperref}
\hypersetup{
  colorlinks = true,
  urlcolor = blue, 
  linkcolor = blue,
  citecolor = blue}
  
\usepackage{natbib}

\begin{document}
	
	\begin{frontmatter}
		\title{A Statistical Perspective on Coreset Density Estimation}
		\runtitle{Coresets}
		
				    
		\author{ 
			\fnms{Paxton}
			\snm{Turner}\ead[label=paxton]{pax@mit.edu},
			\fnms{Jingbo}
			\snm{Liu}\ead[label=jingbo]{jingbol@illinois.edu},
			and
			\fnms{Philippe} \snm{Rigollet}
            \thanks{P.R. was supported by NSF awards IIS-1838071, DMS-1712596, DMS-1740751, and DMS-2022448.}
            \ead[label=rigollet]{rigollet@math.mit.edu}
		    }
		
		\affiliation{Massachusetts Institute of Technology and University of Illinois at Urbana-Champaign }
		
		%
		\address{{Paxton Turner}\\
			{Department of Mathematics} \\
			{Massachusetts Institute of Technology}\\
			{77 Massachusetts Avenue,}\\
			{Cambridge, MA 02139-4307, USA}\\
			\printead{paxton}
		}

		\address{{Jingbo Liu}\\
		    {Department of Statistics} \\
			{University of Illinois at Urbana-Champaign}\\
			{725 S. Wright St.,}\\
			{Champaign, IL 61820, USA} \\
			\printead{jingbo}
		}
		
		\address{{Philippe Rigollet}\\
			{Department of Mathematics} \\
			{Massachusetts Institute of Technology}\\
			{77 Massachusetts Avenue,}\\
			{Cambridge, MA 02139-4307, USA}\\
			\printead{rigollet}
		}

		
		\runauthor{Turner et al.}
		
		\begin{abstract}
	    Coresets have emerged as a powerful tool to summarize data by selecting a small subset of the original observations while retaining most of its information. This approach has led to significant computational speedups but the performance of statistical procedures run on coresets is largely unexplored. In this work, we develop a statistical framework to study coresets and focus on the canonical task of nonparameteric density estimation. Our contributions are twofold. First, we establish the minimax rate of estimation achievable by coreset-based estimators. Second, we show that the practical coreset kernel density estimators are near-minimax optimal over a large class of H\"{o}lder-smooth densities.
		\end{abstract}
		
		
	\end{frontmatter}

\section{Introduction}    
    
    The ever-growing size of datasets that are routinely collected has led practitioners across many fields to contemplate effective data summarization techniques that aim at reducing the size of the data while preserving the information that it contains. While there are many ways to achieve this goal, including standard data compression algorithms, they often prevent direct manipulation of data for learning purposes. \emph{Coresets} have emerged as a flexible and efficient set of techniques that permit direct data manipulation. Coresets are well-studied in machine learning~\citep{HarKus07,FelSchSoh13,BacLucKra17,BacLucKra18,KarLib19}, statistics~\citep{FelFauKra11,ZhePhi17,MunSchSoh18,HugCamBro16,PhiTai18,PhiTai19}, and computational geometry~\citep{AgaHarVar05, Cla10,FraSoh05,GarJag09, ClaGenSol20}. 
	
Given a dataset $\cD=\{X_1, \ldots, X_n\} \subset \mathbb{R}^d$ and task (density estimation, logistic regression, etc.) a coreset $\cC$ is given by $\cC=\{X_i\,:\, i \in S\}$ for some subset $S$ of $\{1, \ldots, n\}$ of size $|S|\ll n$. A good coreset should suffice to perform the task at hand with the same accuracy as with the whole dataset $\cD$.

In this work we study the canonical task of density estimation. Given i.i.d random variables  $X_1, \ldots, X_n \sim \p_f$ that admit a common density $f$ with respect to the Lebesgue measure over $\mathbb{R}^d$, the goal of density estimation is to estimate~$f$. It is well known that the minimax rate of estimation over the $L$-H\"{o}lder smooth densities $\mc{P}_{\mc{H}}(\beta, L)$ of order $\beta$ is given by
\begin{equation}
    \label{eqn:classical_risk}
     \inf_{ \hat{f} } \sup_{ f \in \mc{P}_{\mc{H}}(\beta, L)  }\mb{E}_f \, \norm{ \hat{f} - f }_2    = \Theta_{\beta, d, L}( n^{-\frac{\beta}{2 \beta + d} } )\,,
\end{equation}
where the infimum is taken over all estimators based on the dataset $\cD$. Moreover the minimax rate above is achieved by a kernel density estimator
\begin{equation}
\label{EQ:KDE}
\hat f_n(x):=\frac{1}{nh^d}\sum_{j=1}^n K\left(\frac{X_i-x}{h}\right)
\end{equation}
for suitable choices of kernel  $K: \mathbb{R}^d \to \mathbb{R}$ and bandwidth $h>0$  \citep[see e.g.][Theorem~1.2]{Tsy09}.

The main goal of this paper is to extend this understanding of rates for density estimation to estimators based on coresets. Specifically we would like to characterize the statistical performance of coresets in terms of their cardinality. To do so, we investigate two families of estimators built on coresets: one that is quite flexible and allows arbitrary estimators to be used on the coreset and another that is more structured  and driven by practical considerations; it consists of weighted kernel density estimators built on coresets.  

%


\subsection{Two statistical frameworks for coreset density estimation} 

We formally define a coreset as follows. Throughout this work $m=o(n)$ denotes the cardinality of the coreset. Let $S = S(y|x)$ denote a conditional probability measure on $\binom{[n]}{m}$, where $x \in \mb{R}^{d \times n}$. 
In information theoretic language, $S$ is a channel from $\mb{R}^{d \times n}$ to subsets of cardinality $m$. 
We refer to the channel $S$ as a \textit{coreset scheme} because it designates a data-driven method of choosing a subset of data points. In what follows, we abuse notation and let $S = S(x)$ denote an instantiation of a sample from the measure $S(y|x)$ for $x \in \mb{R}^{d \times n}$. 
A \textit{coreset} $X_S$ is then defined to be the projection of the dataset $X = (X_1, \ldots, X_n)$ onto the subset indicated by $S(X)$: $X_S := \{ X_i \}_{i \in S(X)}$.

The first family of estimators that we investigate is quite general and allows the statistician to select a coreset and then employ an estimator that only manipulates data points in the coreset to estimate an unknown density. To study coresets, it is convenient to make the dependence of estimators on observations more explicit than in the traditional literature. More specifically, a density estimator $\hat f$ based on $n$ observations $X_1, \ldots,X_n \in \mb{R}^d$  is a function $\hat f: \mb{R}^{d \times n} \to L_2(\mb{R}^d)$ denoted by $\hat f[X_1, \ldots, X_n](\cdot)$. Similarly, a \textit{coreset-based estimator} $\hat{f}_S$ is constructed from a coreset scheme $S$ of size $m$ and an estimator (measurable function) $\hat f: \mb{R}^{d \times m} \to L_2(\mb{R}^d)$ on $m$ observations. We enforce the additional restriction on $\hat f$ that for all $y_1, \ldots, y_m \in \mb{R}^d$ and for all bijections $\pi:[m] \to [m]$, it holds that $\hat{f}[y_1, \ldots, y_m](\cdot) = \hat{f}[y_{\pi(1)}, \ldots, y_{\pi(m)}](\cdot)$. Given $S$ and $\hat f$ as above, we define the \textit{coreset-based estimator} $\hat{f}_S: \mathbb{R}^{d \times n} \to L_2(\mb{R}^d)$ to be the function $\hat{f}_S[X]( \cdot ) := \hat{f}[X_S](\cdot): \mb{R}^d \to \mb{R}$. We evaluate the performance of coreset-based estimators in Section \ref{sec:coreset_based} by characterizing their rate of estimation over H\"{o}lder classes.\footnote{Our notion of coreset-based estimators bares conceptual similarity to various notions of \textit{compression schemes} as studied in the literature, e.g. \cite{LitWar86,AshBenHar20,HanKonSad19}.}

The symmetry restriction on $\hat f$ prevents the user from exploiting information about the ordering of data points to their advantage: the only information that can be used by the estimator $\hat f$ is contained in the unordered collection of distinct vectors given by the coreset $X_S$.

As evident from the the results in Section~\ref{sec:coreset_based}, the information-theoretically optimal coreset estimator does not resemble coreset estimators employed in practice. To remedy this limitation, we also study \textit{weighted coreset kernel density estimators} (KDEs) in Section~\ref{sec:Cara}. Here the statistician selects a kernel $k$, bandwidth parameter $h$, and a coreset $X_S$ of cardinality $m$ as defined above and then employs the estimator
\[
\hat f_S(y) = \sum_{j \in S} \lambda_j h^{-d} k\left( \frac{X_j - y}{h} \right),
\]
where the weights $\{ \lambda_j \}_{j \in S}$ are nonnegative, sum to one and are allowed to depend on the full dataset.

In the case of uniform weights where $\lambda_j = \frac{1}{m}$ for all $j \in S$, coreset KDEs are well-studied \citep[see e.g.][]{BacLacObo12,HarSam14,PhiTai18,PhiTai19,KarLib19}. Interestingly, our results show that allowing flexibility in the weights gives a definitive advantage for the task of density estimation. By Theorems \ref{thm:main_Caratheodory} and \ref{thm:kde_lbd}, the uniformly weighted coreset KDEs require a much larger coreset than that of weighted coreset KDEs to attain the minimax rate of estimation over univariate Lipschitz densities.





\subsection{Setup and Notation}

We reserve the notation $\norm{ \cdot }_2$ for the $L_2$ norm
and $\abs{\cdot}_p$ for the $\ell_p$-norm. The constants $c, c_{\beta, d}, c_{L},$ etc. vary from line to line and the subscripts indicate parameter dependences. 

Fix an integer $d \ge 1$. For any  multi-index $s = (s_1,\ldots, s_d) \in \mathbb{Z}_{\geq 0}^d$ and $x =(x_1, \ldots, x_d) \in \mathbb{R}^d$, define $s!=s_1!\cdots s_d!$, $x^s=x_1^{s_1}\cdots x_d^{s_d}$ and let  $D^s$ denote the differential operator defined by
$$
D^s=\frac{\partial^{|s|_1}}{\partial x_1^{s_1} \cdots \partial x_d^{s_d}}\,.
$$
Fix a positive real number $\beta,$ and let $\cbeta$ denote the maximal integer \textit{strictly} less than $\beta$. Given a multi-index $s$, the notation $\abs{s}$ signifies the coordinate-wise application of $\abs{\cdot}$ to $s$.  

Given $L>0$ we let  $\mc{H}(\beta, L)$ denote the space of H\"{o}lder functions $f:\mb{R}^d \to \mb{R}$ that are supported on the cube $[-1/2,1/2]^d$, are $\cbeta$ times differentiable, and satisfy
\begin{equation*}
|D^s f(x) - D^s f(y)| \leq L \abs{x - y}_2^{\beta - \cbeta}\,,\quad 
\end{equation*}
for all $x, y \in \mb{R}^d$ and for all multi-indices $s$ such that $\abs{s}_1 = \flo{\beta}$. 

Let $\mc{P}_{\mc{H}}(\beta, L)$ denote the set of probability density functions contained in $\mc{H}(\beta,L)$. For $f \in \mc{P}_{\mc{H}}(\beta,L)$, let $\mb{P}_f$ (resp. $\E_f$) denote the probability distribution (resp. expectation) associated to $f$.

For $d \geq 1$ and $\gamma \in \mb{Z}_{\geq 0}$, we also define the Sobolev functions $\sd(\gamma, L')$ that consist of all $f: \mb{R}^d \to \mb{R}$ that are $\gamma$ times differentiable and satisfy
\begin{equation*}
\norm{D^\alpha f}_2 \leq L'
\end{equation*} 
for all multi-indices $\alpha$ such that $\abs{\alpha}_1 = \gamma$.


\section{Coreset-based estimators}
\label{sec:coreset_based}

In this section we study the performance of coreset-based estimators. Recall that coreset-based estimators are estimators that only depend on the data points in the coreset. 

Define the \textit{minimax risk for coreset-based estimators} $ \psi_{n,m}(\beta,L)$ over $\cP_{\mc{H}}(\beta, L)$ to be 
\begin{equation}
    \label{eqn:coreset_risk}
    \psi_{n,m}(\beta,L) = \inf_{ \hat{f}, |S| = m } \sup_{ f \in \mc{P}_{\mc{H}}(\beta,L)} \mb{E}_{f} \, \norm{\hat{f}_S - f}_2,
\end{equation}
where the infimum above is over all choices of coreset scheme $S$ of cardinality $m$ and all estimators $\hat{f}:\mb{R}^{d \times m} \to L_2(\mb{R}^d)$. 

Our main result on coreset-based estimators characterizes their minimax risk.

\begin{theorem}
    \label{thm:coreset_rate}
Fix $\beta, L>0$ and an integer $d \geq 1$. Assume that $m = o(n)$. Then the minimax risk of coreset-based estimators satisfies
\begin{equation*}
        \label{eqn:coreset_rate}
        \inf_{ \hat{f}, |S| = m } \, \, \sup_{ f \in \mc{P}_{\mc{H}}(\beta, L)} \mb{E}_{f} \, \norm{\hat{f}_S - f}_2 = \\ \Theta_{\beta, d, L}(  n^{-\frac{\beta}{2\beta + d}} + (m \log{n})^{-\frac{\beta}{d}}   ).  
\end{equation*}
\end{theorem}

The above theorem readily yields a characterization of the minimal size  $m^*(\beta,d)$ that a coreset can have while still enjoying the minimax optimal rate $n^{-\frac{\beta}{2\beta + d}}$ from~\eqref{eqn:classical_risk}. More specifically, let $m^*=m^*(n)$ be such that 
\begin{itemize}
\item[(i)] if $m(n)$ is a sequence such that $m =o( m^*)$, then  $\liminf_{n \to \infty}n^{\frac{\beta}{2\beta + d}}\psi_{n,m}(\beta)= \infty$, and 
\item[(ii)] if $m =\Omega( m^*)$ then $\limsup_{n \to \infty}\psi_{n,m}(\beta)n^{\frac{\beta}{2\beta + d}}\le C_{\beta, d, L}$ for some constant $C_{\beta, d, L}>0$. 
\end{itemize}
Then it follows readily from from Theorem~\ref{thm:coreset_rate} that $m^*= \Theta_{\beta, d, L}( n^{\frac{d}{2 \beta + d}}/\log n)$.

Theorem \ref{thm:coreset_rate} illustrates two different curses of dimensionality: the first stems from the original estimation problem, and the second stems from the compression problem. As $d \to \infty$, it holds that $m^* \sim n/\log n$, and in this regime there is essentially no compression, as the implicit constant in Theorem \ref{thm:coreset_rate} grows rapidly with $d$.\footnote{ In fact, even for the classical estimation problem \eqref{eqn:classical_risk}, this constant scales as $d^d$ \citep[see][Theorem 3]{McD17}.}

Our proof of the lower bound in Theorem \ref{thm:coreset_rate} first uses a standard reduction from estimation to multiple hypothesis testing problem over a finite function class. While Fano's inequality is the workhorse of our second step, note that the lower bound must hold only for coreset estimators and not \emph{any} estimator as in standard minimax lower bounds. This additional difficulty is overcome by a careful handling of the information structure generated by coreset scheme channels rather than using off-the-shelf results for minimax lower bounds. The full details of the lower bound are in the Appendix.

The estimator achieving the rate in Theorem \eqref{thm:coreset_rate} relies on an encoding procedure. It is constructed by building a dictionary between the subsets in $\binom{[n]}{m}$ and an $\varepsilon$-net on the space of H\"{o}lder functions. The key idea is that, for $\omega(1) = m \leq n/2$, the amount of subsets of size $m$ is extremely large, so for $m$ large enough, there is enough information to encode a nearby-neighbor in $L_2(\mathbb{R}^d)$ to the kernel density estimator on the entire dataset. 

\subsection{Proof of the upper bound in Theorem \ref{thm:coreset_rate}}

Fix $\eps = c^*(m \log n)^{-\frac{\beta}{d}}$ for $c^*$ to be determined and let  $\cN_\eps$ denote an $\eps$-net of $\mc{P}_{\mc{H}}(\beta, L)$ with respect to the $L_2([-\frac{1}{2},\frac12]^d)$ norm. It follows from the classical Kolmogorov-Tikhomorov bound~\citep[see, e.g., Theorem XIV of][]{Tik93} that there exists a constant $C_{\mathsf{KT}}(\beta, d, L)>0$ such that we can choose $\cN_\eps$ with $\log |\cN_\eps|\le C_{\mathsf{KT}}(\beta, d, L)\,\eps^{-d/\beta}$. In particular, there exists $\mathsf{f} \in \cN_\eps$ such that $\|\hat f_n - \mathsf{f}\|_{L^2([-1/2,1/2]^d)}\le \eps$ where $\hat f_n$ is the minimax optimal kernel density estimator defined in~\eqref{EQ:KDE}.

We now develop our encoding procedure for $\mathsf{f}$. To that end, fix an integer $K \ge m$ such that $\binom{K}{m}\ge |\cN_\eps|$ and let $\phi : \binom{[K]}{m} \to 
\cN_\eps$ be any surjective map. Our procedure only looks at the first coordinates of the sample $X = \{X_1, \ldots, X_n\}$. Denote these coordinates by $x=\{x_1, \ldots, x_n\}$ and note that these $n$ numbers are almost surely distinct. Let $A$ denote a parameter to be determined, and define the intervals
\begin{equation*}
    B_{ik} = [ (i - 1)K^{-1}A + (k -1)A , \\ \quad (i - 1)K^{-1}A + (k -1)A + K^{-1}A ]. 
\end{equation*}

For $i = 1, \ldots, K$, define 
\begin{equation*}
    B_i = \bigcup_{k =1}^{1/A} B_{ik} . 
\end{equation*}

%

The next lemma, whose proof is in the Appendix, ensures that with high probability every bin $B_i$ contains the first coordinate $x_i$ of at least one data point.

\begin{lemma}
\label{lem:good_bins}    
    Let $K^{-1} = c(\log n)/n$ for $c>0$ a sufficiently large absolute constant, and let $A = A_{\beta, L,K}$ denote a sufficiently small constant. Then for all $f \in \mc{P}_{\mc{H}}(\beta, L)$ and $X_1, \ldots, X_n \stackrel{iid}{\sim} \mb{P}_f$, the event that for every $j = 1, \ldots, K$ there exists some $x_{i}$ in bin $B_j$ holds with probability at least $1 - O(n^{-2})$.
\end{lemma}

In the high-probability event $\cE$ that every bin $B_i$ contains the first coordinate of some data point, choose a unique representative $x_j^\circ \in x$ such that $x_j^\circ\in B_j$ and pick any $T_{\mathsf{f}}\in \phi^{-1}(\mathsf{f})$. Then define $S=\{i\,:\, x_i=x^\circ_j, j \in T_{\mathsf{f}}\}$. If there exists a bin with no observation, then let $X_S$ consist of two data points lying in the same bin and $m - 2$ random data points. Then set $\hat f_S \equiv 0$. 

Note that $\hat f_S$ is indeed a coreset-based estimator. The function $\hat f$ such that $\hat f_S = \hat f[X_S]$ looks at the $m$ data points in the coreset, and if their first coordinates lie in distinct bins, then $X_S$ is decoded as above to output the corresponding element $\mathsf{f}$ of the net $\mc{N}_\eps$. Otherwise, $\hat f \equiv 0$. 

Next, it suffices to show the upper bound of Theorem \ref{thm:coreset_rate} in the case when $m \leq c n^{d/(2\beta+ d)}$ for $c$ a sufficiently small absolute constant. For $c^* = c^*_{\beta, d, L}$ sufficiently large, by Stirling's formula and our choice of $K$ it holds that
\begin{equation*}
    \label{eqn:Stirling}
     \log \binom{K}{m} 
     \geq   C_{\mathsf{KT}}(\beta, d, L) \, \left( \frac{1}{\eps} \right)^{\frac{d}{\beta}} \ge \log |\cN_\eps|.
\end{equation*}
Hence, the surjection $\phi$ and our encoding estimator $\hat f_S$ are well-defined.

Next
we have
$$
 \mb{E}_f \norm{\hat f_S -f }_2=  \mb{E}_f\big[ \norm{\mathsf{f} -f }_2 \1_{\cE}\big]+  \mb{E}_f\big[ \norm{0  -f }_2 \1_{\cE^c}\big].
$$
We control the first term as follows using~\eqref{eqn:classical_risk} and the fact that $\norm{\mathsf{f} - \hat f_n}_2\le \eps$ on $\cE$:
\begin{align*}
\mb{E}_f\big[ \norm{\mathsf{f} -f }_2 \1_{\cE}\big] &\le \mb{E}_f\norm{\hat f_n - f}_2+\mb{E}_f\norm{\mathsf{f} - \hat f_n}_2   \\
&\le c_{\beta, d, L} \,\big(n^\frac{-\beta}{2\beta +d}+(m \log n)^{-\frac{\beta}{d}}\big).
\end{align*} 

By the Cauchy-Schwarz inequality, 
\begin{align*}
\mb{E}_f\big[ \norm{0-f }_2 \1_{\cE^c}\big] &\le \big(\mb{E}_f \norm{f }_2^2 \, \p(\cE^c)\big)^{1/2} \\ 
& \le c_{\beta, d, L} \, n^{-1}\,.
\end{align*}

Put together, the previous three displays yield the upper bound of Theorem \ref{thm:coreset_rate}.

\section{Coreset kernel density estimators}
\label{sec:Cara}

In this section, we consider the family of weighted kernel density estimators built on coresets and study its rate of estimation over the H\"{o}lder densities. In this framework, the statistician first computes a minimax estimator $\hat f$ using the entire dataset and then approximates $\hat f$ with a weighted kernel density estimator over the coreset. Here we allow the weights to be a measurable function over the entire dataset rather than just the coreset.  

As is typical in density estimation, we consider kernels $k: \mb{R}^d \to \mb{R}$ of the form $k(x) = \prod_{i = 1}^d \kappa(x_i)$ where $\kappa$ is an even function and $\int \kappa(x) \, \ud x = 1$. Given bandwidth parameter $h$, we define $k_h(x) = h^{-d} \, k(\frac{x}{h})$.

\subsection{Carath\'{e}odory coreset method}
\label{sec:Cara_construction}
Given a KDE with uniform weights and bandwidth $h$ defined by
    \[
    \hat f(y) = \frac{1}{n} \sum_{j = 1}^n k_h(X_j - y),
    \]
on a sample $X_1, \ldots, X_n$, we define a coreset KDE $\hat g_S$ as follows in terms of a cutoff frequency $T>0$. Define $A = \{ \omega \in \frac{\pi}{2}\mb{Z}^d:\, \abs{\omega}_\infty \leq T \}$. Consider the complex vectors $(e^{i \langle X_j, \omega \rangle} )_{\omega \in A}$. By Carath\'{e}odory's theorem \citep{Car07}, there exists a subset $S \subset [n]$ of cardinality at most $2(1 + \frac{4T}{\pi})^d + 1$ and nonnegative weights $\{ \lambda_j \}_{j \in S}$ with $\sum_{j \in S} \lambda_j = 1$ such that

\begin{equation}
\label{eqn:Cara_condition}
\frac{1}{n} \sum_{j = 1}^n (e^{i \langle X_j,  \omega \rangle})_{\omega \in A} = \sum_{j \in S} \lambda_j (e^{i \langle X_j, \omega \rangle})_{\omega \in A}. 
\end{equation}
Then $\hat g_S(y)$ is defined to be
\begin{equation*}
\hat g_S(y) = \sum_{j \in S} \lambda_j k_h(X_j - y). 
\end{equation*}


\subsubsection{Algorithmic considerations} 
For a convex polyhedron $P$ with vertices $v_1, \ldots, v_n \in \mb{R}^D$, the proof of Carath\'{e}odory's theorem is constructive and yields a polynomial-time algorithm in $n$ and $D$ to find a convex combination of $D+1$ vertices that represents a given point in $P$ \citep{Car07} \citep[see also ][Theorem 1.3.6]{HirLem04}. For completeness, we describe below this algorithm applied to our problem.  Note that, more generally, for a large class of convex bodies, Carath\'{e}odory's theorem may be implemented efficiently using standard tools from convex optimization \citep[][Chapter 6]{GroLovSch12}. 

Set $D = 2|A| \leq 2(1 + \frac{4T}{\pi})^d$. For $j = 1, \ldots, n$, let 
\[
v_j = ( \mathsf{Re} \, e^{i\langle X_j, \omega \rangle}, \mathsf{Im} \, e^{i\langle X_j, \omega \rangle} )_{\omega \in A} \in \mb{R}^{D}. 
\]
Let $M$ denote the matrix with columns $(v_1, 1)^T, \ldots, (v_n,1)^T \in \mb{R}^{D+1}$, and let $\Delta_{n-1} \subset \mb{R}^n$ denote the standard simplex. Assume without loss of generality that $n \geq D + 2$. Next, 

\begin{enumerate}
\item Find a nonzero vector $w \in \text{ker}(M)$
\item Find $\alpha>0$ so that $\lambda_1 := \frac{1}{n} \1 + \alpha w$ lies on the boundary of $\Delta_{n-1}$
\end{enumerate} 
Observe that $M \lambda_1 = ( \frac{1}{n} \sum v_i, 1)^T$, and since $\lambda_1 \in \partial \Delta_{n-1}$ the average is now represented using a convex combination of at most $n - 1$ of the vertices $v_1, \ldots, v_n$. As long as at least $D+2$ vertices remain, we can continue reducing the number of vertices used to represent $\frac{1}{n} \sum v_j$ by applying steps 1 and 2. Thus after at most $n - D-1$ iterations, we obtain $\lambda \in \Delta_D$  that satisfies $\sum \lambda_j v_j = \frac{1}{n} \sum v_i$, as desired. 




\subsection{Results on Carath\'{e}odory coresets} 
\label{sec:Cara_results}
Proposition \ref{prop:Caratheodory_prop} is key to our results and specifies conditions on the kernel guaranteeing that the Carath\'{e}odory method yields an accurate estimator. 

\begin{proposition}
\label{prop:Caratheodory_prop}
Let $k(x) = \prod_{i = 1}^d \kappa(x_i)$ denote a kernel with $\kappa \in \sd(\gamma, L')$ such that $\abs{\kappa(x)} \leq c_{\beta,d} \abs{x}^{-\nu}$ for some $\nu \geq \beta + d$, and the KDE 
\[
\hat f(y) = \frac{1}{n} \sum_{i = 1}^n k_h(X_i - y)
\] with bandwidth $h = n^{-\frac{1}{2\beta + d}}$ satisfies
\begin{equation}
\label{eqn:minimax_condition}
\sup_{f \in \hd(\beta, L)} \mb{E} \norm{f - \hat f}_2 \leq c_{\beta, d, L}  \, n^{-\frac{\beta}{2\beta + d}}.
\end{equation}
Then the Carath\'{e}odory coreset estimator $\hat g_S(y)$ constructed from $\hat f$ with $T = c_{d, \gamma,L' } \, n^{\frac{d/2 + \beta + \gamma}{\gamma(2\beta + d)}}$ satisfies
\begin{equation*}
	\sup_{f \in \hd(\beta, L)} \mb{E} \norm{\hat g_S - f}_2 \leq c_{\beta, d,L} \, n^{-\frac{\beta}{2\beta + d}}.
\end{equation*}
\end{proposition}

There exists a kernel $k_s \in \mc{C}^\infty$ that satisfies the conditions above for all $\beta$ and $\gamma$. We sketch the details here and postpone the full argument to the Proof of Theorem \ref{thm:main_Caratheodory} in the Appendix. Let $\psi:[-1,1] \to [0,1]$ denote a cutoff function that has the following properties: $\psi \in \mc{C}^{\infty}$, $\psi\big|_{[-1, 1]} \equiv 1$, and $\psi$ is compactly supported on $[-2, 2]$. Define $\kappa_S(x) = \mc{F}[\psi](x)$, and let $k_s(x) = \prod_{i = 1}^d \kappa_S(x_i)$ denote the resulting kernel. Observe that for all $\beta > 0$, the kernel $k_s$ satisfies 
\[
\text{ess sup}_{\omega \neq 0} \frac{1 - \mc{F}[k_s](\omega)}{ \abs{\omega}^{\alpha} } \leq 1, \quad \forall  \alpha \preceq \beta. 
\]
Using standard results from \citep{Tsy09}, this implies that the resulting KDE $\hat f_s$ satisfies \eqref{eqn:minimax_condition}. Since $\psi = \mc{F}^{-1}[k_s] \in \mc{C}^\infty$, the Riemann--Lebesgue lemma guarantees that $\abs{\kappa_s(x)} \leq c_{\beta,d} \abs{x}^{\nu}$ is satisfied for $\nu = \ceil{\beta + d}$. Since $\psi$ is compactly supported, an application of Parseval's identity yields $\kappa_s \in \sd(\gamma, c_{\gamma})$. Applying Proposition \ref{prop:Caratheodory_prop} to $k_s$, we conclude that for the task of density estimation, weighted KDEs built on coresets are nearly as powerful as the coreset-based estimators studied in Section~\ref{sec:coreset_based}.

\begin{theorem}
\label{thm:main_Caratheodory}
Let $\eps > 0$. The Carath\'{e}odory coreset estimator $\hat g_S(y)$ built using the kernel $k_s$ and setting $T = c_{d, \beta, \eps} \, n^{\frac{\eps}{d} + \frac{1}{2\beta + d} }$ satisfies  
\[
\sup_{f \in \hd(\beta, L)} \mb{E}_f \norm{\hat g_S - f}_2 \leq c_{\beta, d, L} \, n^{-\frac{\beta}{2\beta + d}}.
\]
The corresponding coreset has cardinality \[m = c_{d, \beta, \eps} n^{ \frac{d}{2\beta + d}+\eps}.\] 
\end{theorem}

Theorem \ref{thm:main_Caratheodory} shows that the Carath\'eodory coreset estimator achieves the minimax rate of estimation with near-optimal coreset size. In fact, a small modification yields a near-optimal rate of convergence for any coreset size as in Theorem~\ref{thm:coreset_rate}.

\begin{corollary}
\label{cor:general_Cara_bd}
Let $\eps > 0$ and $m \leq c_{\beta, d, \eps} \, n^{\frac{d}{2\beta + d} + \eps}$. The Carath\'{e}odory coreset estimator $\hat g_S(y)$ built using the kernel $k_s$, setting $h = m^{-\frac{1}{d} +  \frac{\eps}{\beta} }$ and $T =c_d \, m^{1/d}$, satisfies
\begin{equation*}
\sup_{f \in \hd(\beta, L)} \mb{E}\norm{ \hat g_S - f}_2 \leq c_{\beta, d, \eps, L} \, \left(  m^{-\frac{\beta}{d} + \eps} + n^{-\frac{\beta}{2\beta + d} + \eps} \right),
\end{equation*} 
and the corresponding coreset has cardinality $m$.
\end{corollary}

Next we apply Proposition \ref{prop:Caratheodory_prop} to the popular Gaussian kernel $\phi(x) = (2 \pi)^{-d/2} \exp(-\frac{1}{2}\abs{x}_2^2)$. This kernel has rapid decay in the real domain and Fourier space, and is thus amenable to our techniques. Moreover, $k_\phi$ is a kernel of order $\ell = 1$, \citep[Definition 1.3 and Theorem 1.2]{Tsy09} and so the standard KDE $\hat f_\phi$ on the full dataset attains the minimax rate of estimation $c_{d,L} n^{1/(2 + d)}$ over the Lipschitz densities $\hd(1, L)$.  

\begin{theorem}
Let $\eps> 0$. The Carath\'{e}odory coreset estimator $\hat g_\phi(y)$ built using the kernel $\phi$ and setting $T = c_{d,\eps} \, n^{ \frac{1}{2 + d}+\frac{\eps}{d}}$ satisfies  
\[
\sup_{f \in \hd(1, L)} \mb{E} \norm{\hat g_\phi - f}_2 \leq c_{ d, L} \, n^{-\frac{1}{2 + d}}.
\]
The corresponding coreset has cardinality \[m = c_{d, \eps} n^{\frac{d}{2 + d}+\eps}.\]
\end{theorem}

In addition, we have a nearly matching lower bound to Theorem \ref{thm:main_Caratheodory} for coreset KDEs. In fact, our lower bound applies to a generalization of coreset KDEs where the vector of weights $(\lambda_j)_j$ is not constrained to be in the simplex but can range within a hypercube of width that may grow polynomially with $n$: $\max_{j \in S} \abs{\lambda _j} \leq n^B$.

\begin{theorem}
\label{thm:weight_kde_lbd}
Let $A, B \geq 1$. Let $k$ denote a kernel with $\norm{k}_2 \leq n$. Let $\hat g_S$ denote a weighted coreset KDE with bandwidth $h \geq n^{-A}$ built from $k$ with weights $\{\lambda_j\}_{j \in S}$ satisfying $\max_{j \in S} \abs{\lambda _j} \leq n^B$. Then
\begin{equation*}
\sup_{ f \in \hd(\beta, L)}  \E_f \norm{\hat g_S -  f}_2 \geq \\ c_{\beta, d, L} \left[ (A + B)^{-\frac{\beta}{d}} (m \log{n})^{-\frac{\beta}{d}} + n^{-\frac{\beta}{2\beta + d}} \right].
\end{equation*}
\end{theorem}

This result is essentially a consequence of the lower bound in Theorem~\ref{thm:coreset_rate} because, in an appropriate sense, coreset KDEs with bounded weights are well-approximated by coreset-based estimators. Hence, in the case of bounded weights, allowing these weights to be measurable functions of the entire dataset rather than just the coreset, as would be required in Section~\ref{sec:coreset_based}, does not make a significant difference for the purpose of estimation. The full details of Theorem~\ref{thm:weight_kde_lbd} are postponed to the Appendix. 


\subsection{Proof sketch of Proposition \ref{prop:Caratheodory_prop}}

Here we sketch the proof of Proposition \ref{prop:Caratheodory_prop}, our main tool in constructing effective coreset KDEs. Full details of the argument may be found in the Appendix. 

Let $k(x) = \prod_{i =1}^d \kappa(x_i)$ denote a kernel, and suppose that $\hat f(y) = \frac{1}{n} \sum_{i = 1}^n k_h(X_i -y)$ is a good estimator for an unknown density $f$ in that 
\[\norm{ f - \hat f}_2 \leq \eps := c_{\beta,d } n^{-\frac{\beta}{2\beta + d}}\] on setting $h = n^{-1/(2\beta + d)}$. Our goal is to find a subset $S \subset [n]$ and weights $\{\lambda_j\}_{j \in S}$ such that \[\frac{1}{n} \sum_{i = 1}^n k_h(X_i -y) \approx \sum_{j \in S} \lambda_j k_h(X_j - y)  .\] Suppose for simplicity that $\kappa$ is compactly supported on $[-1/2, 1/2]$. By hypothesis and Parseval's theorem $\kappa \in \sd(\gamma, L')$, and we can further show that $k \in \sd(\gamma, c_{d, L'})$ and $k_h \in \sd(\gamma, c_{d, L'} h^{-d/2 - \gamma})$. Let $\mc{\bar F}$ denote the Fourier transform on the interval $[-1, 1]$. Using the Fourier decay of $k_h$, we have
\begin{equation}
\label{eqn:kernel_series}
\norm{ k_h(x) - \sum_{ \abs{\omega}_\infty < T } \mc{\bar F}[k_h](\omega) e^{i \langle x, \omega \rangle} }_2 \leq \eps 
\end{equation}
when $T = (\frac{c_{d, \gamma, L'} h^{-\frac{d}{2} - \gamma}}{\eps} )^{1/\gamma} = c_{d, \gamma, L'} \, n^{\frac{d/2 + \beta + \gamma}{\gamma(2\beta + d)}}$. Observe that this matches the setting of $T$ in Proposition~\ref{prop:Caratheodory_prop}.  

The approximation \eqref{eqn:kernel_series} implies that for $X_i \in [-1/2, 1/2]^d$, 
\begin{equation*}
\hat f(y) \approx \sum_{ \abs{\omega}_\infty < T } \mc{\bar F}[k_h](\omega) \left( \frac{1}{n} \sum_{i = 1}^n e^{i \langle X_i, \omega \rangle} \right)  e^{i \langle y, \omega \rangle}. 
\end{equation*}
Using the Carath\'{e}odory coreset and weights $\{\lambda_j\}_{j \in S}$ constructed in Section \ref{sec:Cara_construction}, it follows that
\begin{equation*}
\sum_{ \abs{\omega}_\infty < T } \mc{\bar F}[k_h](\omega) \left( \frac{1}{n} \sum_{i = 1}^n e^{i \langle X_i, \omega \rangle} \right)  e^{i \langle y, \omega \rangle} = \\
\sum_{ \abs{\omega}_\infty < T } \mc{\bar F}[k_h](\omega) \left( \sum_{i = 1}^n \lambda_j e^{i \langle X_i, \omega \rangle} \right)  e^{i \langle y, \omega \rangle}.
\end{equation*}
Applying \eqref{eqn:kernel_series} again, we see that the right-hand-side is approximately equal to $\hat g_S(y)$, the estimator produced in Section \eqref{sec:Cara_construction}. By the triangle inequality, we conclude that $\norm{\hat g_S(y) - f}_2 \leq c_{\beta, d} \, \eps$, as desired.

\section{Lower bounds for coreset KDEs with uniform weights} 

In this section we study the performance of univariate uniformly weighted coreset KDEs
\begin{equation*}
\hat f_S^{\mathsf{unif}}(y) = \frac{1}{m} \sum_{i \in S} k_h(X_i - y),
\end{equation*}
where $X_S$ is the coreset and $|S| = m$. The next results demonstrate that for a large class of kernels, there is significant gap between the rate of estimation achieved by $\hat f_S^{\text{unif}}(y)$ and that of coreset KDEs with general weights. First we focus on the particular case of estimating Lipschitz densities, the class $\hd(1, L)$. For this class, the minimax rate of estimation (over all estimators) is $n^{-1/3}$, and this can be achieved by a weighted coreset KDE of cardinality $c_\eps n^{1/3 + \eps}$ by Theorem \ref{thm:main_Caratheodory}, for all $\eps > 0$. 

\begin{theorem}
    \label{thm:kde_lbd}
     Let $k$ denote a nonnegative kernel satisfying
    $$
            k(t) = O( \abs{t}^{-(k+1)}), \quad  \text{and} \quad
            \mc{F}[k](\omega) = O(\abs{\omega}^{-\ell})
    $$
        for some $\ell > 0,\, k > 1$. Suppose that $0<\alpha<1/3$. If 
        \[
        m \leq \frac{n^{\frac23 - 2\left( \alpha(1 - \frac{2}{\ell}) + \frac{2}{3 \ell} \right)}}{\log n},
        \]
        then
        \begin{equation}
            \label{eqn:kde_lbd}
            \inf_{ h, S: |S|  \leq m } \, \sup_{f \in \mc{P}_{\mc{H}}(1, L)} \mathbb{E} \norm{ \unif - f  }_2= \Omega_k\Big( \frac{n^{-\frac13+\alpha}}{\log n} \Big).  
        \end{equation}
        The infimum above is over all possible choices of bandwidth $h$ and all coreset schemes $S$ of cardinality at most $m$.
\end{theorem}

By this result, if $k$ has lighter than quadratic tails and fast Fourier decay, the error in \eqref{eqn:kde_lbd} is a polynomial factor larger than the minimax rate $n^{-1/3}$ when $m \ll n^{2/3}$. Hence, our result covers a wide variety of kernels typically used for density estimation and shows that the uniformly weighted coreset KDE performs much worse than the encoding estimator or the Carath\'{e}odory method. In addition, for very smooth univariate kernels with rapid decay, we have the following lower bound that applies for all $\beta > 0$. 

\begin{theorem}
\label{thm:kde_lbd_allbeta}
Fix $\beta>0$ and a nonnegative kernel $k$ on $\mathbb{R}$ satisfying the following fast decay and smoothness conditions:
\begin{align}
\lim_{s\to+\infty}\frac1{s}\log\frac1{\int_{|t|>s}k(t)dt}&>0,
\\
\lim_{\omega\to\infty}\frac1{|\omega|}\log\frac1{|\mathcal{F}[k](\omega)|}&>0,
\end{align}
where we recall that $\mathcal{F}[k]$ denotes the Fourier transform.
Let $\hat{f}_S^{\mathsf{unif}}$ be the uniformly weighted coreset KDE.
Then there exists $L_{\beta}>0$ such that for $L\ge L_{\beta}$ and any $m$ and $h>0$, we have
\begin{align*}
\inf_{ h, S: |S| \leq  m } \, \sup_{f \in \mc{P}_{\mc{H}}(\beta, L)} \mathbb{E} \norm{ \hat{f}_S^{\mathsf{unif}} - f  }_2
&=\Omega_{\beta, k}\left(\tfrac{m^{-\frac{\beta}{1+\beta}}}{\log^{\beta + \frac{1}{2}} m}\right).
\end{align*}
\end{theorem}
Therefore attaining the minimax rate with $\unif$ requires $m \geq n^{\frac{\beta + 1}{2\beta + 1}}$ for such kernels. Next, note that the Gaussian kernel satisfies the hypotheses of Theorem \ref{thm:kde_lbd} and \ref{thm:kde_lbd_allbeta}. As we show in Theorem \ref{thm:discrepancy_phitai}, results of \citep{PhiTai19} imply that our lower bounds are tight up to logarithmic factors: there exists a uniformly weighted Gaussian coreset KDE of size $m = \tilde{O}(n^{2/3})$ that attains the minimax rate $n^{-1/3}$ for estimating univariate Lipschitz densities ($\beta = 1$). In general, we expect a lower bound $m = \Omega( n^{\frac{\beta + d}{2\beta + d}} )$ to hold for uniformly weighted coreset KDEs attaining the minimax rate. The proofs of Theorems \ref{thm:kde_lbd} and \ref{thm:kde_lbd_allbeta} can be found in the Appendix. 



\section{Comparison to other methods}

Three methods for constructing coreset kernel density estimators that have previously been explored include random sampling \citep{JosKomVar11,LopMuaSch15}, the Frank--Wolfe algorithm \citep{BacLacObo12,HarSam14,PhiTai18}, and discrepancy-based approaches \citep{PhiTai19,KarLib19}. These procedures all result in a uniformly weighted coreset KDE. To compare these results with ours on the problem of density estimation, for each method under consideration we raise the question: How large does $m$, the size of the coreset, need to be to guarantee that
\begin{equation}
\label{eqn:comparison}
\sup_{ f \in \hd(\beta, L)} \mb{E}_f\norm{ \hat g_S - f}_2 = O_{\beta, d, L}\left( n^{-\frac{\beta}{2\beta + d}} \right) \, ?
\end{equation}  
Here $\hat g_S$ is the resulting coreset KDE and the right-hand-side is the minimax rate over all estimators on the full dataset $X_1, \ldots, X_n$. 

Uniform random sampling of a subset of cardinality $m$ yields an i.i.d dataset, so the rate obtained is at least $m^{-\beta/(2\beta + d)}$. Hence, we must take $m = \Omega( n)$ to achieve the minimax rate.

The Frank--Wolfe algorithm is a greedy method that iteratively constructs a sparse approximation to a given element in a convex set \citep{MarWol56, Bub14}. Thus Frank--Wolfe may be applied directly in the RKHS corresponding to a positive-semidefinite kernel as shown in \cite{PhiTai19} to approximate the KDE on the full dataset. However, due to the shrinking bandwidth in our problem, this approach also requires $m = \Omega( n)$ to guarantee the bound in \eqref{eqn:comparison}. Another strategy is to approximately solve the linear equation \eqref{eqn:Cara_condition} using the Frank--Wolfe algorithm. Unfortunately, a direct implementation again uses $m = \Omega(n)$ data points. 

A more effective strategy utilizes discrepancy theory \citep{Phi13, PhiTai19, KarLib19} \citep[see][for a comprehensive exposition of discrepancy theory]{Mat99,Cha00}. By the well-known halving algorithm \citep[see e.g.][]{ChaMat96,PhiTai19} if for all $N \leq n$, the \textit{kernel discrepancy}
\[
\mathsf{disc}_k = \sup_{x_1, \ldots, x_N} \, \min_{\substack{\sigma \in \{-1, +1\}^n \\ \1^T \sigma = 0  }} \norm{ \sum_{i = 1}^N \sigma_i k(x_i - y)  }_\infty
\]
is at most $D$, then there exists a coreset $X_S$ of size $\tilde{O}_D(\eps^{-1})$ such that
\begin{equation}
\label{eqn:discrepancy_bd}
\norm{ \frac{1}{n} \sum_{i = 1}^n k(X_i - y) - \frac{1}{m} \sum_{j \in S} k(X_i - y)}_\infty = \tilde{O}_D(\eps). 
\end{equation}

The idea of the halving algorithm is to maintain a set of datapoints $\mc{C}_\ell$ at each iteration and then set $\mc{C}_{\ell +1}$ to be the set of vectors that receive sign $+1$ upon minimizing $\norm{ \sum_{x \in \mc{C}_\ell } \sigma_x k(x - y) }_\infty$. Starting with the original dataset and repeating this procedure $O(\log \frac{n}{m})$ times yields the desired coreset $X_S$ satisfying \eqref{eqn:discrepancy_bd}.  

\citet[][Theorem 4]{PhiTai19} use a state-of-the-art algorithm from \cite{BanDadGarLov18} called the \textit{Gram--Schmidt walk} to give strong bounds on the kernel discrepancy of bounded and Lipschitz kernels $k: \mb{R}^d \times \mb{R}^d \to \mb{R}$ that are positive definite and decay rapidly away from the diagonal. With a careful handling of the Lipschitz constant and error in their argument when the bandwidth is set to be $h = n^{-1/(2\beta + d)}$, their techniques yield the following result applied to the kernel $k_s$. For completeness we give details of the argument in the Appendix. 

\begin{theorem}
\label{thm:discrepancy_phitai}
Let $k_s$ denote the kernel from Section \ref{sec:Cara_results}. The algorithm of \cite{PhiTai19} yields in polynomial time a subset $S$ with $\abs{S} = m = \tilde{O}( n^{\frac{\beta+d}{2\beta + d}} )$ such that the uniformly weighted coreset KDE $\hat g_S$ satisfies 
\[
\sup_{ f \in \mc{P}_{\mc{H}}(\beta, L)  } \E \norm{f - \hat g_S}_2 \leq c_{\beta, d, L}  \, n^{-\frac{\beta}{2\beta + d}} .
\]
\end{theorem}

This result also applies to more general kernels, for example, the Gaussian kernel when $\beta = 1$. We suspect that this is the best result achievable by discrepancy-based methods. In particular for nonnegative univariate kernels with fast decay in the real and Fourier domains, such as the Gaussian kernel, Theorem \ref{thm:kde_lbd} implies that this rate is optimal for estimating Lipschitz densities with uniformly weighted coreset KDEs. 

In contrast, the \Cara coreset KDE as in Theorem \ref{thm:main_Caratheodory} only needs cardinality $m = O_{\varepsilon}( n^{\frac{d}{2\beta + d} + \varepsilon})$ to be a minimax estimator. By Theorem \ref{thm:weight_kde_lbd}, this result is nearly optimal for coreset KDEs with bounded kernels and weights. And as with the other three methods described, our construction is computationally efficient. Hence allowing more general weights results in more powerful coreset KDEs for the problem of density estimation.

	\appendix
	
	\section{PROOFS FROM SECTION 2}

\subsection{Proof of Lemma \ref{lem:good_bins}}
Here we prove Lemma \ref{lem:good_bins},
restated below for convenience.

\begin{lemma*}
Let $K^{-1} = c(\log n)/n$ for $c>0$ a sufficiently large absolute constant, and let $A = A_{\beta, L,K}$ denote a sufficiently small constant. Then for all $f \in \mc{P}_{\mc{H}}(\beta, L)$ and $X_1, \ldots, X_n \stackrel{iid}{\sim} \mb{P}_f$, the event that for every $j = 1, \ldots, K$ there exists some $x_{i}$ in bin $B_j$ holds with probability at least $1 - O(n^{-2})$.
\end{lemma*}

\begin{proof}
Note that $f_1(x_1) \in \mc{P}_{\mc{H}}(\beta, L)$ as a univariate density because $f(x) \in \mc{P}_{\mc{H}}(\beta, L)$. Hence, $f_1$ satisfies
\[
|f_1(x) - f_1(y)| \leq L |x - y|^\alpha 
\]
for some absolute constants $L > 0$ and $\alpha \in (0,1)$.
If $B_{ik} = B_{jk} + s$ for $s \leq A$, then
\begin{align} 
    \label{eqn:subbin_bd}
    \abs{ \mb{P}( B_{ik} ) - \mb{P}( B_{jk}) } \leq 
    \int_{B_{ik}} \abs{ f(x_1) - f(x_1 + s) }  \ud x_1 \leq L K^{-1} A^{1 + \alpha}.  
\end{align} 
Thus for all $i,j$,
    \begin{align}
    \abs{ \mb{P}(B_i) - \mb{P}(B_j) } \leq \sum_{k = 1}^{1/A} \abs{ \mb{P}( B_{ik} ) - \mb{P}( B_{jk}) } \leq L K^{-1}A^{\alpha}. 
    \end{align}
It follows that for all $i = 1, \ldots, K$, 
\begin{equation}
    \label{eqn:A_limit}
    \lim_{A \to 0} \mb{P}(B_i) = K^{-1}. 
\end{equation} 
    
Let $\mc{E}$ denote the event that every bin $B_i$ contains at least one observation $x_k$. By the union bound,
$$
\p(\cE^c)\le \sum_{j=1}\p(X_{11} \notin B_j)^n \le K \max_j (1 - \mb{P}(B_j))^n.
$$
By \eqref{eqn:A_limit}, choosing $A$ small enough ensures that $\mb{P}[B_j] \geq (1/2)K^{-1}$ for all $j$. In fact, by \eqref{eqn:subbin_bd} one may take $A = (\frac{1}{2K^{-2}L})^{1/\alpha}$. Hence, setting $K^{-1} = c
(\log n)/n$ for $c$ sufficiently large, we have 
\begin{equation*}
    \p(\cE^c) = O(n^{-2}). 
\end{equation*}
    
\end{proof}
\subsection{Proof of the lower bound in Theorem \ref{thm:coreset_rate}}
\label{sec:coreset_rate_pf}

In this section, $X = X_1, \ldots, X_n \in \mb{R}^d$ denotes the sample. It is convenient to consider a more general family of \textit{decorated coreset-based estimators}. A \textit{decorated coreset} consists of a coreset $X_S$ along with a data-dependent binary string $\sigma$ of length $R$. A decorated coreset-based estimator is then given by $\hat f[X_S, \sigma]$, where $\hat f: \mb{R}^{d \times m} \times \{0,1\}^R \to L^2([-1/2, 1/2]^d)$ is a measurable function. As with coreset-based estimators, we require that $\hat f[x_1, \ldots, x_m, \sigma]$ is invariant under permutation of the vectors $x_1, \ldots, x_m \in \mb{R}^d$. We slightly abuse notation and refer to the channel $S: X \to Y_S = (X_S, \sigma)$ as a decorated coreset scheme and $\hat f_S$ as the decorated coreset-based estimator. 
The next proposition implies the lower bound in Theorem \ref{thm:coreset_rate}
on setting $R = 0$, in which case a decorated coreset-based estimator is just a coreset-based estimator. This  more general framework allows us to prove Theorem \ref{thm:coreset_rate} 
on lower bounds for weighted coreset KDEs. 

\begin{proposition}
\label{prop:decorated}
Let $\hat f_S$ denote a decorated coreset-based estimator with decorated coreset scheme $S$ such that $\sigma \in \{0, 1\}^R$. Then
\[
\sup_{ f \in \mc{P}_{\mc{H}}(\beta, L)} \E_f \norm{\hat f_S - f}_2 \geq c_{\beta, d, L} \left( ( m \log n + R )^{-\frac{\beta}{d}} + n^{-\frac{\beta}{2\beta + d}} \right). 
\]

\end{proposition}

\subsubsection{Choice of function class}
\label{subsec:discretize}

Fix $h\in (0,1)$ such that $1/h^d$ is integral to be chosen later. Let $z_1, \ldots, z_{1/h^d}$ label the points in $\{\frac{1}{2}h \cdot \1_d + h \mb{Z}^d \} \cap [-1/2,1/2]^d$, where $\1_d$ denotes the all-ones vector of $\mathbb{R}^d$.
We consider a class of functions of the form $f_\omega(x) = 1 + \sum_{j = 1}^{1/h^d} \omega_j g_j(x)$ indexed by $\omega \in \{0,1\}^{1/h^d}$. Here, $g_j(x)$ is defined to be
\[
g_j(x) = 
h^\beta \phi\left( \frac{x - z_j}{h} \right)
\]
where $\phi:\mb{R}^d \to \mb{R}$ is $L$-H\"{o}lder smooth of order $\beta$, has $\norm{\phi}_\infty = 1$,  and has $\int \phi(x) \, \mathrm{d}x = 0$. 



Informally, $f_\omega$ puts a bump on the uniform distribution with amplitude $h^\beta$ over $z_j$ if and only if $\omega_i = 1$. Using a standard argument~\citep[Chapter~2]{Tsy09} we can construct a packing $\cW$ of $\{0,1\}^{1/h^d}$ which results $\cG=\{f_\omega:\omega \in \cW\}$ of the function class $\{f_\omega:\omega \in \{0,1\}^{1/h^d}\}$ such that
\begin{itemize}
    \item[(i)] $\|f-g\|_2\ge c_{\beta, d, L} \, h^{\beta}$ for all $f,g \in \cG, f \neq g$ and,
    \item[(ii)] $\cG$ is large in the sense that $M:=|\cG|\ge 2^{c_{\beta, d,L}/h^d}$.
\end{itemize}

\subsubsection{Minimax lower bound}
\label{appendix:coreset_lower_bound}
 Using standard reductions from estimation to testing, we obtain that
\begin{align}
    \inf_{ \substack{\hat{f}, |S| = m, \\ \sigma \in \{0, 1\}^R } } \, \sup_{ f \in \mc{P}_{\mc{H}}(\beta,L)} \mb{E}_{f} \, \norm{\hat{f}_S - f}_2 &\geq \inf_{ \substack{\hat{f}, |S| = m, \\ \sigma \in \{0, 1\}^R } } \, \max_{ f \in \mc{G} } \, \mb{E}_{f} \, \norm{\hat{f}_S - f}_2 \nonumber \\
    &\geq c_{\beta, d,L} \, h^{\beta} \cdot \inf_{ \psi_S }\frac{1}{M}\sum_{\omega \in \cV}\p_{f_\omega}[ \psi_S(X)\neq \omega ].    \label{eqn:minimax_redux}
\end{align}
where the infimum in the last line is over all tests $\psi_S: \mb{R}^{d \times n} \to [M]$ of the form $\psi_S(X) = \psi(Y_S)$ for a decorated coreset scheme $S$ and a measurable function $\psi:\mathbb{R}^{d \times m} \times \{0, 1\}^R \to [M]$.

Let $V$ denote a random variable that is distributed uniformly over $\cV$ and observe that
$$
\frac{1}{M} \sum_{ \omega \in \cV} \mb{P}_{f_\omega}[ \psi_S(X) \neq \omega ]=\p[\psi_S(X) \neq V]
$$
where $\p$ denotes the joint distribution of $(X,V)$ characterized by the conditional distribution $X|V=\omega$ which is assumed to have density $f_\omega$ for all $\omega \in \cV$.

Next, by Fano's inequality~\citep[Theorem~2.10.1]{CovTho06} and the chain rule, we have
\begin{equation}
    \label{eqn:fano1}
  \p[\psi_S(X) \neq V]\ge 1-\frac{I(V;\psi_S(X))+1}{\log M}\,,
\end{equation}
where   $I(V;\psi_S(X))$ denotes the mutual information between $V$ and $\psi_S(X)$ and we used the fact that the entropy of $V$ is $\log M$. Therefore, it remains to control $I(V;\psi_S(X))$. To that end, note that it follows from the data processing inequality that
\begin{equation*}
    I(V;\psi_S(X)) \le I(V; (X_S, \sigma) ) = I(V; Y_S) =\mathsf{KL}(P_{V,Y_S}\|P_V\otimes P_{Y_S})\,,
\end{equation*}
where $P_{V,Y_S}, P_V$ and $P_{Y_S}$ denote the distributions of $(V,Y_S)$, $V$ and $Y_S$ respectively and observe that $P_{Y_S}$ is the mixture distribution given by
$P_{Y_S}(A, t)=M^{-1}\sum_{\omega \in \cV} P_{f_\omega}(X_S \in A, \sigma = t)$ for $A \subset \mb{R}^{d \times m}$ and $t \in \{0, 1\}^R$. Denote by $f_{\omega, Y_S}$ the mixed density of $P_{f_\omega}(X_S \in \cdot, \sigma = \cdot)$, where the continuous component is with respect to the Lebesgue measure on $[-1/2,1/2]^{d\times m}$. Denote by $\bar f_{Y_S}$ the mixed density of the uniform mixture of these:
$$
\bar f_{Y_S}:=\frac1M\sum_{\omega \in \cV}f_{\omega, Y_S}\,.
$$

By a standard information-theoretic inequality, for all measures $\mb{Q}$ it holds that 
\begin{equation}
\label{eqn:standard_info}
\mathsf{KL}(P_{V,Y_S}\|P_V\otimes P_{Y_S}) = \frac{1}{M} \sum_{\omega} \mathsf{KL}( P_{Y_S | \omega} \| \, P_{Y_S}  )  \leq \frac{1}{M} \sum_{\omega} \mathsf{KL}( P_{Y_S | \omega} \|\, \mb{Q}  ).
\end{equation}
In fact, we have equality precisely when $\mb{Q} = P_{Y_S}$, and \eqref{eqn:standard_info} follows immediately from the nonnegativity of the KL-divergence. Setting $\mb{Q} = \mathsf{Unif}[-\frac12,\frac12]^d \otimes \mathsf{Unif}\{ 0, 1\}^R $, for all $\omega$ we have
\begin{align}
{\sf KL}(P_{Y_S|\omega}, \mb{Q} ) &= \sum_{t \in \{0, 1\}^R} \, \int_{ [-\frac12,\frac12]^d } f_{\omega, Y_S}(x, t) \log \frac{f_{\omega, Y_S}(x, t)}{ 2^{-R} } \, \ud x \nonumber \\
&\le\sum_{t \in \{0, 1\}^R} \int_{ [-\frac12,\frac12]^d } f_{\omega, Y_S}(x, t) \log f_{\omega, Y_S}(x, t) \, \ud x  + R \label{eqn:klbd}.
\end{align} 
Our next goal is to bound the first term on the right-hand-side above.

\begin{lemma}
    \label{lem:KL_bound}
 For any $\omega \in \cV$, we have
    \[
 \sum_{t \in \{0, 1\}^R} \int_{ [-\frac12,\frac12]^d } f_{\omega, Y_S}(x, t) \log f_{\omega, Y_S}(x, t) \, \ud x \leq 3 m \log{n}.   
    \]
\end{lemma}
\begin{proof}

Let $\p_{X_S}$ denote the distribution of the (undecorated) coreset $X_S$, and note that the density of this distribution is given by $f_{\omega, X_S}(x) := \sum_{t \in \{0, 1\}^R} f_{\omega, Y_S}(x,t)$. Then because the logarithm is increasing,
\begin{align}
\sum_{t \in \{0, 1\}^R} \int_{ [-\frac12,\frac12]^d } f_{\omega, Y_S}(x, t) \log f_{\omega, Y_S}(x, t) \, \ud x &\leq 
\sum_{t \in \{0, 1\}^R} \int_{ [-\frac12,\frac12]^d } f_{\omega, Y_S}(x, t) \log f_{\omega, X_S}(x) \, \ud x \nonumber \\
&= \int_{ [-\frac12,\frac12]^d } f_{\omega, X_S}(x) \log f_{\omega, X_S}(x) \, \ud x \nonumber.
\end{align}

By the union bound,
$$
\p_{X_S}(\cdot)\le \sum_{s \in \binom{[n]}{m}}\p_{X_s}(\cdot)= \binom{n}{m} \p_{X_{[m]}}(\cdot)\,.
$$
It follows readily that $f_{\omega, X_S}(\cdot) \le \binom{n}{m} f_{\omega, X_{[m]}}(\cdot)$\,. Next, let $Z \in [-1/2,1/2]^{d \times m}$ be a random variable with density $f_{\omega, X_S}$ and note that
$$
\int f_{\omega, X_S}\log f_{\omega, X_S} =\E\log  f_{\omega, X_S}(Z)\le \log  \binom{n}{m} + \E \log  f_{\omega, X_{[m]}}(Z)\le m \log \big(\frac{en}{m}\big)+m \log 2\,,
$$
 where in the last inequality, we use the fact that $ f_{\omega, X_{[m]}}=f_\omega^m \le 2^m$. The lemma follows. 
\end{proof}
Since $\log M \ge c_{\beta, d, L} h^{-d}$, it follows from \eqref{eqn:fano1}--\eqref{eqn:klbd} and Lemma \ref{lem:KL_bound} that 
\[
\p[\psi_S(X) \neq V]\ge 1-\frac{3 m \log n + R+1}{\log M} \geq 0.5
\]
on setting $h=c_{\beta, d, L} (m\log n + R)^{-1/d}$. Plugging this value back into~\eqref{eqn:minimax_redux} yields 
$$
 \inf_{ \hat{f}, |S| = m } \sup_{ f \in \mc{P}_{\mc{H}}(\beta,L)} \mb{E}_{f} \, \norm{\hat{f}_S - f}_2 \ge c_{\beta, d, L} (m \log n + R)^{-\beta/d}\,.
$$
Moreover, it follows from standard minimax theory \citep[see e.g.][Chapter 2]{Tsy09} that
$$
 \inf_{ \hat{f}, |S| = m } \sup_{ f \in \mc{P}_{\mc{H}}(\beta,L)} \mb{E}_{f} \, \norm{\hat{f}_S - f}_2 \ge c_{\beta, d, L} n^{-\frac{\beta}{2\beta +d}}\,.
$$
Combined together, the above two displays give the lower bound of Proposition \ref{prop:decorated}.

\section{PROOFS FROM SECTION 3} 


\subsection{Proof of Proposition \ref{prop:Caratheodory_prop}}


We restate the result below. 

\begin{proposition*}
Let $k(x) = \prod_{i = 1}^d \kappa(x_i)$ denote a kernel with $\kappa \in \sd(\gamma, L')$ such that $\abs{\kappa(x)} \leq c_{\beta,d} \abs{x}^{-\nu}$ for some $\nu \geq \beta + d$, and the KDE 
\[
\hat f(y) = \frac{1}{n} \sum_{i = 1}^n k_h(X_i - y)
\] with bandwidth $h = n^{-\frac{1}{2\beta + d}}$ satisfies
\begin{equation*}
\sup_{f \in \hd(\beta, L)} \mb{E} \norm{f - \hat f}_2 \leq c_{\beta, d, L}  \, n^{-\frac{\beta}{2\beta + d}}.
\end{equation*}
Then the Carath\'{e}odory coreset estimator $\hat g_S(y)$ constructed from $\hat f$ with $T = c_{d, \gamma,L' } \, n^{\frac{d/2 + \beta + \gamma}{\gamma(2\beta + d)}}$ satisfies
\begin{equation*}
	\sup_{f \in \hd(\beta, L)} \mb{E} \norm{\hat g_S - f}_2 \leq c_{\beta, d,L} \, n^{-\frac{\beta}{2\beta + d}}.
\end{equation*}
\end{proposition*}

Let $\varphi:\mb{R}^d \to [0,1]$ denote a cutoff function that has the following properties: $\varphi \in \mc{C}^{\infty}$, $\varphi\big|_{[-1, 1]^d} \equiv 1$, and $\varphi$ is compactly supported on $[-2, 2]^d$.

\begin{lemma}
\label{lem:cutoff_kernel}
	Let $\tilde{k}_h(x) = k_h(x) \varphi(x)$ where $\abs{\kappa(x)} \leq c_{\beta, d} \abs{x}^{-\nu}$. Then
	\[
		\norm{\tilde{k}_h - k_h}_2 \leq c_{\beta, d} \, h^{-d + \nu} .
	\]
\end{lemma} 

\begin{proof}

    \begin{align*}
    	\norm{\tilde{k}_h - k_h}_2 &= \norm{ (1 - \varphi) k_h  }_2 \\
    	&\leq \norm{(1 - \1_{[-1, 1]^d}) k_h}_2 \\
    	&= h^{-d/2} \norm{(1 - \1_{[-\frac{1}{h}, \frac{1}{h}]^d}) k}_2 \\
        &\leq d h^{-d/2} \norm{\1_{ \abs{x_1} \geq \frac{1}{h} } \, k  }_2 \\
        &\leq c_{\beta, d} \, h^{-d/2} \sqrt{ \int_{\abs{x_1} \geq \frac{1}{h}} \kappa^2(x_1) \, \ud x_1 } \\
        &\leq c_{\beta, d} \, h^{-d + \nu}. 
    \end{align*}
%
%
\end{proof}

The triangle inequality and the previous lemma yield the next result.

\begin{lemma}
\label{lem:cutoff_kde}
	Let $k$ denote a kernel such that $\abs{\kappa(x)} \leq c_{\beta, d} \abs{x}_2^{-\nu}$. Recall the definition of $\ti k_h$ from Lemma \ref{lem:cutoff_kernel}. Let $X_1, \ldots, X_m \in \mb{R}^d$, and let 
	\[
	\hat g_S(y) = \sum_{j \in S} \lambda_j k_h(X_j - y)
	\]
	denote where $\lambda_j \geq 0$ and $\1^T \lambda = 1$. Let
	\[
	\tilde{g_S}(y) = \sum_{j \in S} \lambda_j \tilde{k}_h(X_j - y).
	\]
	Then
	\[
	\norm{\hat g_S - \tilde{g_S}}_2 \leq c_{\beta, d} h^{-\nu + d}. 
	\]
\end{lemma} 

Next we show that $\tilde{k}_h$ is well approximated by its Fourier expansion on $[-2, 2]^d$. Since $\tilde{k}_h$ is a smooth periodic function on $[-2, 2]^d$, it is expressed in $L_2$ as a Fourier series on $\frac\pi2 \mb{Z}^d$. Thus we bound the tail of this expansion. In what follows, $\alpha \in \mb{Z}^d_{\geq 0}$ is a multi-index and
     \[
     \mc{\bar F}[f](\omega) = \frac{1}{4^{2d}} \int f(x) e^{i \langle x, \omega \rangle} \, \ud x
     \]
     denotes the (rescaled) Fourier transform on $[-2, 2]^d$, where $\omega \in \frac\pi2 \mb{Z}^d$. 
 
 \begin{lemma}
 \label{lem:kernel_expansion}
 	Suppose that the kernel $k \in \sd(\beta, L')$. Let $A = \{ \omega \in \frac\pi2 \mb{Z}^d: \, \abs{\omega}_1 \leq T \}$, and define
 	\[
 	\tilde{k}_h^T(y) = \sum_{\omega \in A} \mc{\bar F}[\tilde{k}_h](\omega) e^{i \langle y , \omega \rangle}.
 	\]
 	Then 
 	\[
 	\norm{ (\tilde{k}_h - \tilde{k}_h^T)\1_{[-2,2]^d} }_2 \leq c_{\gamma, d, L'} \, T^{-\gamma} h^{-d/2-\gamma} 
 	\]
 	
 \end{lemma} 
 
 \begin{proof}
     Observe that for $\omega \notin A$, it holds that
     	\[
     	\sum_{\abs{\alpha}_1 = \gamma} \frac{\gamma!}{\alpha!} \abs{\omega}^\alpha = ( \abs{\omega_1} + \cdots + \abs{\omega_d} )^{\gamma} \geq T^\gamma. 
     	\]
     
     Therefore, 	
     	\begin{align}
     	\norm{\mc{\bar F}[\tilde{k}_h](\omega) \1_{\omega \notin A}}_{\ell_2} 
     	&\leq  \, T^{-\gamma} \norm{ \sum_{\abs{\alpha}_1 = \gamma} \frac{\gamma!}{\alpha!} \abs{\omega}^\alpha \mc{\bar F}[\tilde{k}_h](\omega) \1_{\omega \notin A} }_{\ell_2} \nonumber \\
     	&\leq  \, T^{-\gamma} \sum_{\abs{\alpha}_1 = \gamma} \frac{\gamma!}{\alpha!} \, \norm{  \omega^\alpha \mc{\bar F}[\tilde{k}_h](\omega) }_{\ell_2} \nonumber \\
     	&=  \, c_d \, T^{-\gamma} \sum_{\abs{\alpha}_1 = \gamma} \frac{\gamma!}{\alpha!} \, \norm{ \frac{\partial^\alpha}{\partial x^{\alpha}} \tilde{k}_h(x) }_2, \label{eqn:four_1}
     	\end{align} 
        where in the last line we used Parseval's identity. For any multi-index $\alpha$ with $\abs{\alpha}_1 = \gamma$, 
        \begin{align}
        \norm{ \frac{\partial^\alpha}{\partial x^{\alpha}} \tilde{k}_h(x) }_2
        &= \norm{ \sum_{ \eta \preceq \alpha} \frac{\partial^\eta}{\partial x^{\eta}} k_h(x) \, \frac{\partial^{\alpha-\eta}}{\partial x^{\alpha-\eta}} \varphi(x) }_2 \nonumber \\
        &\leq h^{-\frac{d}{2} - \gamma} \sum_{ \eta \preceq \alpha} c_{d, \gamma} \,  \norm{ \frac{\partial^\eta}{\partial x^{\eta}} k(x) }_2 \label{eqn:four_2},
        \end{align}
        where we used that the derivatives of $\varphi$ are bounded. Next by Parseval's identity,
        \begin{align}
        \norm{ \frac{\partial^\eta}{\partial x^{\eta}} k(x) }_2^2 &= 
        \prod_{i = 1}^d \norm{ \omega_i^{\eta_i} \mc{F}[\kappa](\omega_i) }_2^2. \label{eqn:four_3}
        \end{align}
        For $0 \leq a \leq \gamma$, we have
        \begin{align}
        \int \abs{ \omega^a \mc{F}[\kappa](\omega)}^2  \ud \omega 
        \leq \norm{k}_1 + \int_{\abs{\omega} \geq 1} \abs{ \omega^\gamma \mc{F}[\kappa](\omega)}^2  \ud \omega 
                \leq \norm{k}_1 + L'. \label{eqn:four_4}
        \end{align}
        By \eqref{eqn:four_1}--\eqref{eqn:four_4},
        \begin{align*}
        \norm{\mc{\bar F}[\tilde{k}_h](\omega) \1_{\omega \notin A}}_{\ell_2} \leq c_{d, \gamma, L'} \, T^{-\gamma} h^{-\frac{d}{2} - \gamma}, 
        \end{align*}
        as desired. 
        
\end{proof}

Applying the previous lemma and linearity of the Fourier transform, we have the next corollary that gives an expansion for a general KDE on the smaller domain $[-\frac12, \frac12]^d$. 

\begin{corollary}
	Let $\tilde{g_S}$ denote the KDE built from $\ti k_h$ from Lemma \ref{lem:cutoff_kde} where $X_1, \ldots, X_m \in [-\frac12, \frac12]^d$ and moreover $\kappa \in \sd(\beta, L')$. Let $A = \{ \omega \in \frac{\pi}{2}\mb{Z}^d: \, \abs{\omega}_1 \leq T \}$, and define
	\[
	\tilde{g_S}^T(y) = \sum_{\omega \in A} \mc{\bar F}[ \tilde{g_S} ](\omega) e^{i \langle y, \omega \rangle}. 
	\] 
	Then
	\[
	\norm{ (\tilde{g_S} - \tilde{g_S}^T)\1_{[-\frac12, \frac12]^d} }_2 \leq c_{d, \gamma, L'} \, T^{-\gamma} h^{-d/2-\gamma} L. 
	\]
\end{corollary}

Now we have all the ingredients needed to prove Proposition \ref{prop:Caratheodory_prop}. 

\begin{proof}[Proof of Proposition \ref{prop:Caratheodory_prop}
]
	Let 
	\[
	\tilde{f}(y) = \frac{1}{n} \sum_{j = 1}^n \tilde{k}_h(X_j - y), 
	\]
	and
	\[
	\tilde{g_S}(y) = \sum_{j \in S} \lambda_j \tilde{k}_h(X_j - y). 
	\]
	Also consider their expansions $\tilde{f}^T$ and $\tilde{g}^T$ as defined in Lemma \ref{lem:kernel_expansion}. Observe that, by construction of the Carath\'{e}odory coreset,
	\[
	\tilde{f}^T(y) = \tilde{g}^T(y) \quad \forall y \in [-\frac12, \frac12]^d.
	\] 
	In what follows, $\norm{ \cdot }_2$ is computed on $[-\frac12, \frac12]^d$. By the triangle inequality,
	\begin{align}
		\norm{ \hat g_S - \hat f }_2 &\leq 
		\norm{ \hat g_S - \tilde{g} }_2 + \norm{ \tilde{g} - \tilde{g}^T}_2 + 
		\norm{ \tilde{g}^T - \tilde{f}^T} \nonumber \\
		&\quad + \norm{ \tilde{f}^T - \tilde{f}}_2 +
		\norm{\tilde{f} - \hat f}_2 \nonumber \\
		&\leq c_{\beta, d} \, h^{-d + \nu} + c_{d, \gamma, L'} \, T^{-\gamma} h^{-d/2 - \gamma}  + 0 \nonumber \\
		&\quad + c_{d, \gamma, L'} \, T^{-\gamma} h^{-d/2 - \gamma} + c_{\beta, d} \, h^{-d + \nu} \label{eqn:Cara_main_bound}
	\end{align}
	On the right-hand-side of the first line, the first and last terms are bounded via Lemma \ref{lem:cutoff_kde}. The second and fourth terms are bounded via Lemma \ref{lem:kernel_expansion}, and the third term is $0$ by Carath\'{e}odory. By our choice of $T$ and the decay properties of $k$, we have
	\[
	   \norm{ \hat g_S - \hat f }_2 \leq c_{\beta,  d, L} \, h^\beta \leq c_{\beta,  d, L} \, n^{-\beta/(2\beta + d)}.  
	\]
	The conclusion follows by the hypothesis on $k$, the previous display, and the triangle inequality.
\end{proof}

\subsection{Proof of Theorem 2}
\label{sec:thm2_pf}
We restate Theorem \ref{thm:main_Caratheodory}
here for convenience.

\begin{theorem*}
Let $\eps > 0$. The Carath\'{e}odory coreset estimator $\hat g_S(y)$ built using the kernel $k_s$ and setting $T = c_{d, \beta, \eps} \, n^{\frac{\eps}{d} + \frac{1}{2\beta + d} }$ satisfies  
\[
\sup_{f \in \hd(\beta, L)} \mb{E}_f \norm{\hat g_S - f}_2 \leq c_{\beta, d, L} \, n^{-\frac{\beta}{2\beta + d}}.
\]
The corresponding coreset has cardinality \[m = c_{d, \beta, \eps} n^{ \frac{d}{2\beta + d}+\eps}.\] 
\end{theorem*}

\begin{proof}
Our goal is to apply Proposition \ref{prop:Caratheodory_prop}
to $k_s$. First we show that the standard KDE built from $k_s$ attains the minimax rate on $\hd(\beta, L)$. The Fourier condition 
\[
\text{ess sup}_{\omega \neq 0} \frac{1 - \mc{F}[k_s](\omega)}{ \abs{\omega}^{\alpha} } \leq 1, \quad \forall  \alpha \preceq \beta,
\]
implies that $k_s$ is a kernel of order $\beta$ \citep[][Definition 1.3]{Tsy09}. Since $\mc{F}[k_s](0) = 1 = \int k_s(x) \, \ud x$, it remains to show that the `moments' of order at most $\beta$ of $k_s$ vanish. In fact all of the moments vanish. 
We have, expanding the exponential and using the multinomial formula,
\begin{align*}
\psi(\omega) &= \mc{F}^{-1}[k_s](\omega) \\ 
&= \int k_s(x) e^{-i \langle x, \omega \rangle} \ud x \\
&= \sum_{t = 0}^\infty \int k_s(x)  \frac{(-i \langle x, \omega \rangle)^t}{t!} \ud x \\
&= \sum_{t = 0}^\infty \sum_{\abs{\alpha}_1 = t} \frac{-i^t}{\alpha!} w^\alpha \left\{ \int k_s(x)  x^\alpha \ud x \right\}.
\end{align*}
Since $\psi(\omega) \equiv 1$ in a neighborhood near the origin, it follows that all of the terms $\int k_s(x)  x^\alpha \ud x = 0$. Thus $k_s$ is a kernel of order $\beta$ for all $\beta \in \mb{Z}_{\geq 0}$, and the standard KDE on all of the dataset with bandwidth $h = n^{-1/(2\beta + d)}$ attains the rate of estimation $n^{-\beta/(2\beta + d)}$ over $\hd(\beta, L)$ \citep[see e.g.][Theorem 1.2]{Tsy09}. 

Next, $\abs{\kappa_s(x)} \leq c_{\beta,d} \abs{x}^{\nu}$ for $\nu = \ceil{\beta + d}$. This is because
\[
x^\nu \kappa_s(x) = x^\nu \mc{F}[\psi](x) = \mc{F}\left[\frac{\ud^\nu}{\ud x^{\nu}}\psi\right](x) \leq \norm{ \frac{\ud^\nu}{\ud x^{\nu}}\psi }_1 \leq c_{\beta, d}.
\]
Moreover for all $\gamma \in \mb{Z}_{>0}$, $\kappa_s \in \sd( \gamma, c_\gamma)$. By Parseval's identity,
\[
\norm{ \frac{\ud^\gamma}{\ud x^{\gamma}} \kappa_s }_2 = 
\norm{ \mc{F}[ \frac{\ud^\gamma}{\ud x^{\gamma}} \kappa_s ] }_2 = 
\norm{ \omega^\gamma \psi(\omega)  }_2 \leq c_{\gamma}
\]
because $\psi$ has compact support \citep[see e.g.][Chapter VI]{Kat04}.  

All of the hypotheses of Proposition \ref{prop:Caratheodory_prop}
are satisfied, so we apply the result with 
\[
\gamma = \frac{d(\frac{d}{2} + \beta)}{\eps(2\beta + d)} 
\]
to derive Theorem \ref{thm:main_Caratheodory}.

\end{proof}

\subsection{Proof of Corollary \ref{cor:general_Cara_bd}}


\begin{corollary*}
Let $\eps > 0$ and $m \leq c_{\beta, d, \eps} \, n^{\frac{d}{2\beta + d} + \eps}$. The Carath\'{e}odory coreset estimator $\hat g_S(y)$ built using the kernel $k_s$, setting $h = m^{-\frac{1}{d} +  \frac{\eps}{\beta} }$ and $T =c_d \, m^{1/d}$, satisfies
\begin{equation*}
\sup_{f \in \hd(\beta, L)} \mb{E}\norm{ \hat g_S - f}_2 \leq c_{\beta, d, \eps, L} \, \left(  m^{-\frac{\beta}{d} + \eps} + n^{-\frac{\beta}{2\beta + d} + \eps} \right),
\end{equation*} 
and the corresponding coreset has cardinality $m$.
\end{corollary*}

\begin{proof}

Recall from the proof of Theorem \ref{thm:main_Caratheodory}
that $k_s$ is a kernel of all orders. By a standard bias-variance trade-off \citep[see e.g.][Section 1.2]{Tsy09}, it holds that for the KDE $\hat f$ with bandwidth $h$ (on the entire dataset)

\begin{equation}
\label{eqn:std_KDE_h}
\E_f \norm{\hat f - f}_2 \leq c_{\beta, d, L} \left( h^\beta + \frac{1}{\sqrt{n h^d}}\right). 
\end{equation}
Moreover, from \eqref{eqn:Cara_main_bound} applied to $k_s$ , setting $T = m^{1/d}$, we get
\begin{equation}
\label{eqn:cara_main_h}
\norm{\hat g_S - \hat f}_2 \leq c_{\beta, d} \, h^{\beta} + c_{d, \gamma} \, m^{-d/\gamma} h^{-d/2 -\gamma}. 
\end{equation}
Choosing 
\[
\gamma = (\beta + \frac{d}{2})(\frac{\beta }{d \eps} - 1), \quad h = m^{-\frac{1}{d} + \frac{\eps}{\beta}}
\]
(assuming without loss of generality that $\eps > 0$ is sufficiently small so that $\gamma > 0$), then the triangle inequality, \eqref{eqn:std_KDE_h}, \eqref{eqn:cara_main_h}, and the upper bound on $m$ yield the conclusion of Corollary \ref{cor:general_Cara_bd}.

\end{proof}

\subsection{Proof of Theorem \ref{thm:weight_kde_lbd}}

For convenience, we restate Theorem \ref{thm:weight_kde_lbd}
here.

\begin{theorem*}
Let $A, B \geq 1$. Let $k$ denote a kernel with $\norm{k}_2 \leq n$. Let $\hat g_S$ denote a weighted coreset KDE with bandwidth $h \geq n^{-A}$ built from $k$ with weights $\{\lambda_j\}_{j \in S}$ satisfying $\max_{j \in S} \abs{\lambda _j} \leq n^B$. Then
\begin{equation*}
\sup_{ f \in \hd(\beta, L)}  \E_f \norm{\hat g_S -  f}_2 \geq \\ c_{\beta, d, L} \left[ (A + B)^{-\frac{\beta}{d}} (m \log{n})^{-\frac{\beta}{d}} + n^{-\frac{\beta}{2\beta + d}} \right].
\end{equation*}
\end{theorem*}

\begin{proof}
Let $\lambda = \lambda_1, \ldots, \lambda_m$ and let $\ti \lambda = \ti \lambda_1, \ldots, \ti \lambda_m$. Observe that
\begin{align}
\norm{  \sum_{j \in S} \lambda_j k_h(X_j - y) - 
 \sum_{j \in S} \ti \lambda_j k_h(X_j - y) }_2  
&\leq  \sum_{j \in S} \abs{\lambda_j - \ti \lambda_j} \norm{k_h(X_j - y)}_2 \nonumber \\
&\leq \abs{\lambda - \ti \lambda}_\infty n^2 h^{-d/2}. \label{eqn:lambda_Lip}
\end{align}

Using this we develop a decorated coreset-based estimator $\hat f_S$ (see Section \ref{sec:coreset_rate_pf}) that approximates $\hat g_S$ well. Set $\delta = c_{\beta, d, L} n^{-4} h^{d/2}$ for $c_{\beta, d, L}$ sufficiently small and to be chosen later. Order the points of the coreset $X_S$ according to their first coordinate. This gives rise to an ordering $\preceq$ so that 
\[
X'_1 \preceq X'_2 \preceq \cdots \preceq X'_m
\]
denote the elements of $X_S$. Let $\lambda \in \mb{R}^m$ denote the correspondingly reordered collection of weights so that
\[
\hat g_S(y) =  \sum_{j = 1}^m \lambda_j k_h(X'_j - y). 
\]

Construct a $\delta$-net  $\mc{N}_\delta$  with respect to the sup-norm $\abs{\cdot}_\infty$ on the set $\{ \nu \in \mb{R}^m: \abs{\nu}_\infty \leq n^{B} \}$. Observe that
\begin{equation}
\label{eqn:log_card}
\log \abs{ \mc{N}_\delta } = \log  (n^B \delta^{-1})^m
= c_{\beta, d, L} \, (B + A) m \log n 
\end{equation}
Define $R$ to be the smallest integer larger than the right-hand-side above. Then we can construct a surjection $\phi: \{0, 1 \}^{R} \to \mc{N}_\delta$. Note that $\phi$ is constructed before observing any data: it simply labels the elements of the $\delta$-net $\mc{N}_\delta$ by strings of length $R$. 

Given $\hat g_S(y) =  \sum_{j \in S} \lambda_j k_h(X_j -y)$, define $\hat f_S$ as follows:
\begin{enumerate}
\item Let $\ti \lambda \in \mb{R}^m$ denote the closest element in $\mc{N}_\delta$ to $\lambda \in \mb{R}^m$. 
\item Choose $\sigma \in \{0, 1\}^{R}$ such that $\phi(\sigma) = \ti \lambda$.
\item Define the decorated coreset $Y_S = (X_S, \sigma)$. 
\item Order the points of $X_S$ by their first coordinate. Pair the $i$-th element of $\ti \lambda$ with the $i$-th element $X_i'$ of $X_S$, and define
\[
\hat f_S(y) = \sum_{j = 1}^m \ti \lambda_j k_h(X_j' - y)
\]
\end{enumerate}

We see that $\hat f_S$ is a decorated-coreset based estimator because in step 4 this estimator is constructed only by looking at the coreset $X_S$ and the bit string $\sigma$. Moreover, by \eqref{eqn:lambda_Lip} and the setting of $\delta$,
\begin{equation}
\label{eqn:deco_estimate}
\norm{ \hat f_S - \hat g_S}_2 \leq c_{\beta, d, L} \, n^{-2}. 
\end{equation}
By Proposition \ref{prop:decorated} and our choice of $R$, 
\[
\sup_{ f \in \mc{P}_{\mc{H}}(\beta, L)} \E_f \norm{\hat f_S - f}_2 \geq c_{\beta, d, L} \left( (A+B)^{-\frac{\beta}{d}} ( m \log n )^{-\frac{\beta}{d}} + n^{-\frac{\beta}{2\beta + d}} \right).
\]
Applying the triangle inequality and \eqref{eqn:deco_estimate} yields Theorem \ref{thm:weight_kde_lbd}.
\end{proof}

\section{PROOFS FROM SECTION 4}

\textit{Notation:} Given a set of points  $X = x_1, \ldots, x_m \in [-1/2, 1/2] $ (not necessarily a sample), we let 
\begin{equation*}
\hat f_X(y) = \frac{1}{m} \sum_{i = 1}^m k_h(X_i - y) 
\end{equation*}
denote the uniformly weighted KDE on $X$. 

\subsection{Proof of Theorem \ref{thm:kde_lbd}}

\begin{theorem*}
    Let $k$ denote a nonnegative kernel satisfying
$$
        k(t) = O( \abs{t}^{-(k+1)}), \quad  \text{and} \quad
        \mc{F}[k](\omega) = O(\abs{\omega}^{-\ell})
$$
    for some $\ell > 0,\, k > 1$. Suppose that $0<\alpha<1/3$. If 
    \[
    m \leq \frac{n^{\frac23 - 2\left( \alpha(1 - \frac{2}{\ell}) + \frac{2}{3 \ell} \right)}}{\log n},
    \]
    then
    \begin{equation*}
        \inf_{ h, S: |S|  \leq m } \, \sup_{f \in \mc{P}_{\mc{H}}(1, L)} \mathbb{E} \norm{ \unif - f  }_2= \Omega_k\Big( \frac{n^{-\frac13+\alpha}}{\log n} \Big).  
    \end{equation*}
    The infimum above is over all possible choices of bandwidth $h$ and all coreset schemes $S$ of cardinality at most $m$.
\end{theorem*}

The proof of Theorem \ref{thm:kde_lbd}
follows directly from Propositions \ref{prop:small_h} and \ref{prop:large_h}, which are presented in Sections \ref{sec:small_h} and \ref{sec:large_h}, respectively. 

\subsubsection{Small bandwidth} 
\label{sec:small_h}

First we show that uniformly weighted coreset KDEs on $m$ points poorly approximate densities that are very close to $0$ everywhere. 

\begin{lemma}
\label{lem:step_lbd}
Let $\hat{f}_X$ denote a uniformly weighted coreset KDE built from an even kernel $k:\mb{R} \to \mb{R}$ with bandwidth $h$ on $m$ points $X = x_1, \ldots, x_m \in \mb{R}$. Suppose that quantiles $0 \leq q_1 \leq q_2$ satisfy
    \begin{align}
        \int_{-q_1}^{q_1} k(t) \ud t &\geq 0.9, \label{eqn:q1} \quad \quad \text{and} \\
        \int_{-q_2}^{q_2} k(t) \ud t &\geq 1 - \gamma. \label{eqn:q2}
    \end{align} 
    
Let $U$ denote an interval $[0, u]$ where
\begin{equation}
    \label{eqn:q2_vs_u}
    u \geq 8q_2h,
\end{equation}
and suppose that $f:U \to \mb{R}$ satisfies 
    \begin{equation}
    \label{eqn:f_bd}
    \frac{1}{100 q_1 mh} \leq f(x) \leq \frac{45}{44} \cdot \frac{1}{100 q_1 mh}  
    \end{equation}
    for all $x \in U$. 
    
    Then
    \begin{equation*}
        \inf_{X: |X| = m} \norm{(\hat{f}_X - f)\1_U}_1 \geq \frac{u}{440 q_1 mh} - \gamma. 
    \end{equation*}
\end{lemma}

\begin{proof}
Let $N$ denote the number of $x_i \in X$ such that $[x_i - q_1h, x_i + q_1h] \subset [0, u]$. The argument proceeds in two cases. With foresight, we set $\alpha = 1/(44q_1)$. Also let $C_1 = 1/(100q_1)$ and $C_2 = 45/(4400q_1)$. \newline 

\noindent \textit{Case 1:} $N \geq \frac{\alpha u}{h}$. Then by \eqref{eqn:q1} and the nonnegativity of $k$,  
\[
\norm{\hat{f}_X \1_U}_1 \geq \frac{0.9 N}{m} \geq \frac{0.9 \alpha u}{mh}. 
\]
By \eqref{eqn:f_bd}, 
\[
\norm{f}_1 \leq \frac{C_2 u}{mh}. 
\]
Hence, 
\begin{equation*}
   \norm{(\hat{f}_X - f)\1_U}_1 \geq \frac{u}{mh}(0.9 \alpha  - C_2)
   = C_2 \frac{u}{ mh} = \frac{45}{4400} \cdot \frac{u}{q_1 mh}. 
\end{equation*}
Thus Lemma \ref{lem:step_lbd} holds in Case 1 where $N \geq \alpha u/h$. \newline 

\noindent \textit{Case 2:} $N \leq \frac{\alpha u}{h}$.
Let
\[
V = [2hq_2, u - 2hq_2 ] \, \backslash \bigcup_{j \in T} [ x_j - q_1h ,  x_j + q_1h]
\]
where $T$ is the set of indices $j$ so that $[x_j - q_1h, x_j + q_1h] \subset U$. Observe that if $j \notin T$, then by \eqref{eqn:q2},
\[
\int_{V} \frac{1}{h} k\left( \frac{x_j - t}{h} \right) \ud t \leq 
\gamma. 
\]
If $j \in T$, then by \eqref{eqn:q1}, 
\[
\int_{V} \frac{1}{h} k\left( \frac{x_j - t}{h} \right) \ud t \leq 0.1. 
\]
Thus,
\begin{equation*}
\norm{\hat{f}_X \1_V}_1 \leq \frac{0.1 N}{m} + \gamma
    \leq \frac{\alpha 0.1 u}{mh} + \gamma.
\end{equation*}
By the union bound, observe that the Lebesgue measure of $V$ is at least 
\[u - 4hq_2 - 2Nhq_1 \geq \frac{u}{2} - 2Nhq_1 \geq u( \frac{1}{2} - 2\alpha q_1 ). 
\]Next, by \eqref{eqn:f_bd},
\[
\norm{f \1_V}_1 \geq C_1\frac{u}{mh}(\frac{1}{2} - 2\alpha q_1). 
\]
Therefore,
\begin{equation}
    \label{eqn:case2_L1}
    \norm{ (\hat{f}_X -f)\1_U }_1 \geq \frac{u}{mh}( C_1(1/2 - 2\alpha q_1) -0.1\alpha  ) -  \gamma =  \frac{u}{440 q_1 mh} - \gamma. 
\end{equation}
\end{proof}

\begin{proposition}
    \label{prop:small_h}
    Let $L > 2$. Let $0< \delta <1/3$ denote an absolute constant. Let $\hat{f}_X$ denote a uniformly weighted coreset KDE with bandwidth $h$ built from a kernel $k$ on $X = x_1, \ldots, x_m$. Suppose that $k(t) \leq \Delta |t|^{-(k+1)}  $ for some absolute constants $\Delta > 0, k\geq 1$. If $h \leq n^{-1/3 + \delta}$, then for 
    \[
    m \leq \frac{n^{2/3 - 2 \delta }}{\log n}
    \] it holds that
    \begin{equation}
        \label{eqn:small_h} 
        \sup_{f \in \mc{P}_{\mc{H}}(1,L)} \, \inf_{ X: |X| = m } \, \norm{\hat{f}_X - f}_2 = \Omega\left( \frac{n^{-1/3 + \delta}}{\log n} \right).  
    \end{equation}
\end{proposition}

\begin{proof}
    Let 
    \[f(t) = \lambda \left(  e^{-1/t}\1(t \in [-1/2,0]) + e^{-1/(1 - t)} \1(t \in [0, 1/2]) \right),\]
    where $\lambda$ is a normalizing constant so that $\int f = 1$. Observe that $f \in \mc{P}_{\mc{H}}(1,L)$. Our first goal is to show that 
    \[
    \norm{\hat{f}_X - f}_1 = \Omega\left(\frac{1}{mh \log^2(mh)}\right) 
    \]
    holds for all $\tau/h \leq m \leq h^{-2}$ and for all $h \leq n^{-1/3 + \delta}$, where $\tau$ is an absolute constant to be determined.  
    
    We apply Lemma \ref{lem:step_lbd} to the density $f$.  Let $q_1$ be defined as in Lemma \ref{lem:step_lbd}, and set $C_1 = 1/(100q_1)$ and $C_2 = 45/(4400q_1)$. Set $\tau = 10C_2/\lambda$. Let 
    \[
    U = [t_1, t_2] :=  \left[ \frac{1}{\log(\lambda mh/C_1)} , \, \, \frac{1}{\log(\lambda mh/C_2)} \right].
    \]
    The function $f|_U$ satisfies the bounds \eqref{eqn:f_bd} from Lemma \ref{lem:step_lbd}. Observe that the length of $U$ is
    \[
    u:= t_2 - t_1 = \Omega(\frac{1}{\log^2(mh)}). 
    \]We set the parameter $\gamma$ in Lemma \ref{lem:step_lbd} to be 
    \[
        \gamma = \frac{1}{800 q_1 mh \log^2(mh)}. 
    \]
    By the decay assumption on $k$, we may set
    \[
    q_2 := \left( \frac{2 \Delta}{k\gamma} \right)^{1/k}. 
    \]
    Therefore, 
    \begin{align} 
    u - 8q_2h &= \Omega( \frac{1}{\log^2(mh)}) - 8h \left( \frac{2 \Delta}{k\gamma} \right)^{1/k}\\
    &= \Omega( \frac{1}{\log^2(mh)} ) - O( h (mh \log^2(mh))^{1/k} ) \\
    &= \Omega( \frac{1}{\log^2(h^{-1})} ) - O( h^{1 - 1/k} \log^2(h^{-1}) ) > 0
    \end{align} 
    for $n$ sufficiently large, because we assume $\tau/h \leq m \leq h^{-2}$, $h \leq n^{-1/3+\delta}$, and $k>1$. Hence, condition \eqref{eqn:q2_vs_u} is satisfied for $m, h$ in the specified range, so we apply Cauchy--Schwarz and Lemma \ref{lem:step_lbd} to conclude that for all $\tau/h \leq m \leq h^{-2}$ and $h \leq n^{-1/3+\delta}$,
    \begin{equation}
    \label{eqn:small_h_bd}
    \norm{\hat{f}_X -f}_2 \geq \norm{\hat{f}_X - f}_1 = \Omega\left(\frac{1}{mh \log^2(mh)}\right) = \Omega\left(\frac{1}{mh \log^2(h^{-1})}\right). 
    \end{equation}
    
    Suppose first that $\log^{2}(1/h) \geq n^{1/3 - \delta}$. Then clearly the right-hand side of \eqref{eqn:small_h_bd} is $\Omega(1)$ for $m \leq n$. Otherwise, we have for all $h \leq n^{-1/3 + \delta}$ that if $m$ is in the range 
    \[
    \frac{\tau}{h} \leq m \leq \min\left( \frac{n^{1/3 - \delta} \log n}{h \log^2(1/h)}, \, \,  h^{-2} \right) =: N_h ,
    \]
    then \eqref{eqn:small_h_bd} implies
    \begin{equation}
        \label{eqn:smallh_main}
        \norm{\hat{f}_X - f}_2 = \Omega\left( \frac{n^{-1/3+\delta}}{\log n} \right). 
    \end{equation}
     Moreover, a uniformly weighted coreset KDE on $m = O(1/h)$ points can be expressed as a uniformly weighted coreset KDE on $\Omega( 1/h )$ points by setting some of the $x_i$'s to be duplicates. Hence \eqref{eqn:smallh_main} holds for all $1 \leq m \leq N_h$. Since $N_h$ is a decreasing function of $h$, it follows that \eqref{eqn:smallh_main} holds for all $m \leq n^{2/3 - 2\delta}/\log n$ and $h \leq n^{-1/3 + \delta}$, as desired. 
    
    
\end{proof}

\subsubsection{Large bandwidth}
\label{sec:large_h}

\begin{lemma}
\label{lem:large_h}
Let $\eps = \eps(n) > 0$, and let $\hat{f}_X$ denote the uniformly weighted coreset KDE on $X$ with bandwidth $h$. Suppose that $\phi:\mb{R} \to \mb{R}$ is an odd $\mc{C}^{\infty}$ function supported on $[-1/4, 1/4]$. Let $f(t):[-1/2, 1/2] \to \mb{R}_{\geq 0}$ denote the density 
\[
f(t) = \frac{12}{11}(1 - t^2) + \eps \phi(t) \cos\left( \frac{t}{\eps} \right). 
\]

Then 
\begin{multline}
    \label{eqn:oversmth_gen}
    \norm{ \hat{f}_X - f }_2^2 \geq \frac{1}{2} \eps^2 \left(\norm{\phi}_2^2 - \abs{\mc{F}[\phi^2](2 \eps^{-1})} \right) \\ - \norm{\phi}_1 \sup_{ |\omega| \geq h \eps^{-1}/2} \abs{ \mc{F}[k](\omega)  } -  2 \eps \int_{ |\omega| \geq \eps^{-1}/2 } \abs{ \mc{F}[\phi](\omega)} \ud \omega. 
\end{multline}
\end{lemma}

\begin{proof}
Let $g(t) = (12/11)(1 - t^2)$ and $\psi(t) = \eps \phi(t) \cos(t/\eps)$. 
    Observe that
    \begin{align}
        \norm{ \hat{f}_X - f }_2^2 &\geq \norm{g - f}_2^2 - 2 \langle \hat{f}_X , g - f \rangle +2 \langle g, \psi(t) \rangle \nonumber  \\
        &= \norm{g - f}_2^2 - 2 \langle \hat{f}_X, g - f \rangle 
        \label{eqn:L2_sq} 
    \end{align}
    because $g(t) \psi(t)$ is an odd function.
    Next, using $\cos^2(\theta) = (1/2) (\cos(2\theta) + 1)$, 
    \begin{align}
        \label{eqn:rho_norm}
        \norm{g - f}_2^2 &= \eps^2 \int_{-1/2}^{1/2} \cos^2(t/\eps) \phi^2(t) \ud t \nonumber \\
        &\geq \frac{\eps^2}{2} \norm{\phi}_2^2 - \frac{\eps^2}{2}\abs{ \mc{F}[\phi^2](2\eps^{-1})}. 
    \end{align}
    
    By the triangle inequality and Parseval's formula,
    \begin{equation*}
        \frac{\abs{ \langle \hat{f}_X, g - f \rangle}}{\eps} 
        \leq  \Big( \underbrace{\int_{|\omega| \leq h\eps^{-1}/2}}_{=:A} + \underbrace{\int_{|\omega| \geq h\eps^{-1}/2}}_{=:B} \Big) \Big| \mc{F}[k] \left(- \frac{h}{\eps} - \omega \right) \frac{1}{h} \mc{F}[\phi]\left( -\frac{\omega}{h} \right)\Big| \, \ud\omega. 
    \end{equation*}
    
    Moreover,
    \begin{align}
    \label{eqn:low_freq}
        A &\leq 
        \frac{1}{2\eps} \norm{ \phi }_1 \cdot \sup_{ |\omega| \geq h \eps^{-1}/2} \abs{ \mc{F}[k](\omega)  }, \\
        \label{eqn:high_freq} B 
        &\leq \norm{k}_1 \cdot \int_{ |\omega| > \eps^{-1}/2 } \abs{ \mc{F}[\phi](\omega)} \ud \omega. 
    \end{align}
    Then \eqref{eqn:oversmth_gen} follows from $\norm{k}_1 = 1$ and equations \eqref{eqn:L2_sq}, \eqref{eqn:rho_norm}, \eqref{eqn:low_freq}, and \eqref{eqn:high_freq}. 
\end{proof}

\begin{proposition}
\label{prop:large_h}
Let $\eps = n^{-1/3 + \gamma}$ for some absolute constant $\gamma > 0$. Let $\hat{f}_X$ denote a uniformly weighted coreset KDE with bandwidth $h$ built from a kernel $k$ on $X = x_1, \ldots, x_m$. Suppose that $\abs{\mc{F}[k](\omega)} \leq |\omega|^{-\ell}$. If $h \geq c \eps^{1 - 2/\ell} = cn^{(-1/3 + \gamma)(1 - 2/\ell)}$ for $c$ sufficiently large, then for all $m$ it holds that
\begin{equation}
    \label{eqn:big_h_lbd}
     \sup_{f \in \mc{P}_{\mc{H}}(\beta,L)} \, \inf_{X: |X| = m  }  \norm{ \hat{f}_X - f  }_2 = \Omega( \eps ) = \Omega\left( n^{-1/3 + \gamma} \right)
\end{equation}

\end{proposition}

\begin{proof}
The proof is a direct application of Lemma \ref{lem:large_h}. Let $f(t) = g(t) + \eps \phi(t) \cos(t/\eps)$, where we set 
\begin{equation*}
    \phi(t) =  - e^{ \frac{1}{x(x + 1/4)} }\1(x \in [-1/4, 0]) + e^{ - \frac{1}{x(x - 1/4)} }\1(x \in [0, 1/4]). \nonumber
\end{equation*} 

Observe that $\phi$ is odd and $\phi \in \mc{C}^\infty$. Thus, $\phi^2 \in \mc{C}^\infty$, so by the Riemann--Lebesgue lemma \citep[see e.g.][Chapter VI]{Kat04}, $\mc{F}[\phi^2](\eps^{-1}) \leq 10 \eps$. 
Using a similar argument and noting that $\mc{F}[\phi](\omega) =\omega^{-2}\mc{F}[\phi''](\omega)\le 10\omega^{-3}$, we obtain
\[
    \int_{\abs{\omega} \geq 2 \eps^{-1}} \abs{ \mc{F}[\phi](\omega)} \ud \omega \leq 100 \eps^2. 
\]
Also $\norm{\phi}_2 \geq c'$ for a small absolute constant, and $\norm{\phi}_1 \leq 2$.

Thus Lemma \ref{lem:large_h}, the hypothesis on $k$, and $h \geq c' \eps^{1 - 2/\ell}$ imply that 
\[
\norm{\hat{f}_X - f}_2^2 \geq \frac{c^2}{2}\eps^2 - 2\left(\frac{\eps}{h}\right)^\ell - 200 \eps^3 = \Omega(\eps^2).    
\]
Since $f \in \hd(1,L)$, the statement of the lemma follows.
\end{proof}

\subsection{Proof of Theorem \ref{thm:kde_lbd_allbeta}}


\begin{theorem*}
Fix $\beta>0$ and a nonnegative kernel $k$ on $\mathbb{R}$ satisfying the following fast decay and smoothness conditions:
\begin{align}
\lim_{s\to+\infty}\frac1{s}\log\frac1{\int_{|t|>s}k(t)dt}&>0,
\label{e_decay} 
\\
\lim_{\omega\to\infty}\frac1{|\omega|}\log\frac1{|\mathcal{F}[k](\omega)|}&>0,
\label{e_smooth}
\end{align}
where we recall that $\mathcal{F}[k]$ denotes the Fourier transform.
Let $\hat{f}_S^{\mathsf{unif}}$ be the uniformly weighted coreset KDE.
Then there exists $L_{\beta}>0$ such that for $L\ge L_{\beta}$ and any $m$ and $h>0$, we have
\begin{align*}
\inf_{ h, S: |S| \leq  m } \, \sup_{f \in \mc{P}_{\mc{H}}(\beta, L)} \mathbb{E} \norm{ \hat{f}_S^{\mathsf{unif}} - f  }_2
&=\Omega_{\beta, k}\left(\tfrac{m^{-\frac{\beta}{1+\beta}}}{\log^{\beta + \frac{1}{2}} m}\right).
\end{align*}
\end{theorem*}
\begin{proof}
We follow a similar strategy to the proof of Theorem \ref{thm:kde_lbd}
by handling the cases of small and large bandwidth separately. 

Let $q_1=q_1(k)>0$ be the minimum number such that $\int_{|t|>q_1}k(t)dt\le0.1$.
By the assumption in the theorem, there exists $a>0$ such that 
\begin{align*}
\int_{|t|>s}k(t)dt
\le\frac1{a}\exp(-as),
\quad\forall s\ge 0.
\end{align*}
Note that we can set $L_{\beta}^{(1)}$ large such that for any $\delta\in [0,1]$, there exists $f\in\mathcal{P}_{\mathcal{H}}(\beta,L_{\beta}^{(1)})$ such that $f(x)=\delta$ for $x\in[0,1/2]$.
We first show that for any given $m$ and $h$, we have
\begin{align}
\inf_{S: |S| \leq  m } \, \sup_{f \in \mc{P}_{\mc{H}}(\beta, L_{\beta}^{(1)})} \mathbb{E} \norm{ \hat{f}_S^{\mathsf{unif}} - f  }_1
\ge
0.2\left(1\wedge\frac1{100q_1mh}\right)1\left\{h\le \frac{0.02a}{\log\left(\frac{mq_1}{0.001a}\vee \frac{10}{a}\right)}\wedge 1\right\}.
\label{e_step1}
\end{align}
Let $f$ be an arbitrary function in $f\in\mathcal{P}_{\mathcal{H}}(\beta,L_{\beta}^{(1)})$ such that 
\begin{align*}
f(x)=1\wedge \frac{1}{100q_1mh},
\quad \forall x\in[0,1/2].
\end{align*}
Let $T$ be the set of $i\in S$ for which $x_i\in[q_1h,1/2-q_1h]$.

\noindent \textit{Case 1:} $|T|\ge m\left(1\wedge \frac{1}{100q_1mh}\right)$. Since $k \geq 0$, we have
\begin{align*}
\|\hat{f}_X1_{[0,1/2]}\|_1
\ge \frac{0.9|T|}{m}
\ge 0.9\left(1\wedge \frac{1}{100q_1mh}\right).
\end{align*}
On the other hand,
\begin{align*}
\|f1_{[0,1/2]}\|_1
\le \frac1{2}\left(1\wedge \frac{1}{100q_1mh}\right),
\end{align*}
therefore,
\begin{align*}
\|(\hat{f}_X-f)1_{[0,1/2]}\|_1
\ge 
0.4\left(1\wedge \frac{1}{100q_1mh}\right).
\end{align*}

\noindent \textit{Case 2:}   $|T|< m\left(1\wedge \frac{1}{100q_1mh}\right)$. 
Define 
\begin{align*}
\gamma:=0.1
\left(1\wedge \frac{1}{100q_1mh}\right)
\end{align*}
and 
\begin{align*}
q_2:=\frac{0.02}{h}.
\end{align*}
Note that to verify \eqref{e_step1} we only need to consider the event of $h\le \frac{0.02a}{\log\left(\frac{mq_1}{0.001a}\vee \frac{10}{a}\right)}\wedge 1$,
in which case
\begin{align*}
\int_{|t|>q_2}k(t)dt
&\le \frac1{a}\exp(-aq_2)
\\
&\le
\frac1{a}\cdot 
\left(\frac{0.001a}{mq_1}\wedge 0.1a\right)
\\
&\le\frac1{a}\cdot 
\left(\frac{0.001a}{q_1mh}\wedge 0.1a\right)
\\
&=0.1(1\wedge \frac1{100q_1mh})
\\
&=\gamma.
\end{align*}
Moreover since $\gamma\le 0.1$ we see that $q_2\ge q_1$.
Now define
\begin{align*}
V:=[2hq_2,1/2-2hq_2]\setminus 
\bigcup_{j\in T}[x_j-q_1h,x_j-q_1h].
\end{align*}
Then for $j\notin T$, we have 
\begin{align*}
\int_V\frac1{h}k\left(\frac{x_j-t}{h}\right)dt
\le \gamma
\end{align*}
while for $j\in T$ we have 
\begin{align*}
\int_V\frac1{h}k\left(\frac{x_j-t}{h}\right)dt
\le 0.1.
\end{align*}
Thus,
\begin{align*}
\|\hat{f}_X1_V\|_1
\le
\frac{0.1|T|}{m}
+\gamma
\le 0.2\left(1\wedge \frac{1}{100q_1mh}\right).
\end{align*}
On the other hand, by the union bound we see that the Lebesgue measure of $V$ is at least
\begin{align*}
\frac1{2}-4q_2h-2q_1h|T|
\ge 0.5-4q_2h-0.02
\ge 0.4
\end{align*}
where we used the fact that $q_2h= 0.02$.
Then
\begin{align*}
\|f1_V\|_1\ge 0.4
\left(1\wedge \frac{1}{100q_1mh}\right)
\end{align*}
and hence
\begin{align*}
\|(\hat{f}_X-f)1_{[0,1/2]}\|_1
\ge
\|(\hat{f}_X-f)1_V\|_1
\ge 0.2\left(1\wedge \frac{1}{100q_1mh}\right).
\end{align*}
This concludes the proof of \eqref{e_step1}.

The second step is to show that for given $m$ and $h$, we have
\begin{align}
\inf_{S: |S| \leq  m } \, \sup_{f \in \mc{P}_{\mc{H}}(\beta, L)} \mathbb{E} \norm{ \hat{f}_S^{\mathsf{unif}} - f  }_1
\geq
\frac1{4}
\left(\frac{b(h\wedge 1)}{\log m}\right)^{\beta}
-\frac1{bm^2}
\label{e_step2}
\end{align}
sufficiently large $m$ and $L$ to be determined later, and $0<b<\infty$ is such that 
\begin{align*}
\mathcal{F}[k](\omega)
\le \frac1{b}\exp(-b\omega),\quad
\forall \omega\in\mathbb{R}
\end{align*}
whose existence is guaranteed by the assumption of the theorem.
Let $\phi$ be a smooth, even, nonnegative function supported on $[-1/2,1/2]$ satisfying $\int_{[-1/2,1/2]}\phi=1$.
Define 
\begin{align*}
f_{\epsilon}(t):=\phi(t)\left(c_{\epsilon}+\epsilon^{\beta}\sin \frac{t}{\epsilon}\right)
\end{align*}
where $c_{\epsilon}>0$ is chosen so that $\int_{[-1/2,1/2]}f_{\epsilon}=1$. 
Then $\lim_{\epsilon\to0}c_{\epsilon}=1$,
and in particular $f_{\epsilon}\ge 0$ when $\epsilon<\epsilon(\phi,\beta)$ for some $\epsilon(\phi,\beta)$.
Moreover we can find $L_{\beta}^{(2)}<\infty$ such that $f_{\epsilon}\in \mathcal{P}_{\mathcal{H}}(\beta,L_{\beta}^{(2)})$ for all $\epsilon<\epsilon(\phi,\beta)$.
Now
\begin{align}
\|f_{\epsilon} -\hat{f}_X \|_1
&\ge
|\mathcal{F}[f_{\epsilon}](1/\epsilon)-\mathcal{F}[\hat{f}_X](1/\epsilon)
|
\nonumber \\
&\ge 
\left|\int_{[-1/2,1/2]}f_{\epsilon}(t)e^{-it/\epsilon}dt\right|-\left|\mathcal{F}[k](\frac{h}{\epsilon})\right|
\nonumber \\
&\ge \left|\int_{[-1/2,1/2]}f_{\epsilon}(t)\sin\frac{t}{\epsilon}dt\right|
-
\left|\mathcal{F}[k](\frac{h}{\epsilon})\right|
\nonumber \\
&=\epsilon^{\beta}\left|\int_{[-1/2,1/2]}
\phi(t)\sin^2\frac{t}{\epsilon} dt
\right|
-
\left|\mathcal{F}[k](\frac{h}{\epsilon})\right|
\label{e_even}
\end{align}
where \eqref{e_even} used the fact that $\phi$ is even.
Since $\lim_{\epsilon\to0}\int_{[-1/2,1/2]}
\phi(t)\sin^2\frac{t}{\epsilon} dt=\frac{1}{2}$,
there exists $\epsilon'(\phi)$ such that 
\begin{align*}
\int_{[-1/2,1/2]}
\phi(t)\sin^2\frac{t}{\epsilon} dt
\ge \frac1{4}
\end{align*} 
for any $\epsilon\le \epsilon'(\phi)$.
Now define
\begin{align*}
\epsilon''(h,m)=\frac{b(h\wedge 1)}{2\log m}.
\end{align*}
There exists $m(\phi,\beta,b)<\infty$ such that $\sup_{h>0}\epsilon''(h,m)<\epsilon(\phi,\beta)\wedge \epsilon'(\phi)$
whenever $m\ge m(\phi,\beta,b)$.
With the choice of $\epsilon=\epsilon''(h,m)$, we can continue lower bounding \eqref{e_even} as (for $m\ge m(\phi,\beta,b)$):
\begin{align*}
\frac1{4}
\left(\frac{b(h\wedge 1)}{\log m}\right)^{\beta}
-\frac1{bm^2}.
\end{align*}

Finally, we collect the results for step 1 and step 2.
First observe that the main term in the risk in step 1 can be simplified as 
\begin{align}
&\left(1\wedge\frac1{100q_1mh}\right)1\left\{h\le \frac{0.02a}{\log\left(\frac{mq_1}{0.001a}\vee \frac{10}{a}\right)}\wedge 1\right\}
\nonumber\\
&=\frac1{100q_1mh}\wedge1\left\{\mathcal{A}\right\}
\label{e_step1_s}
\end{align}
where $\mathcal{A}$ denotes the event in the left side of \eqref{e_step1_s}.

Thus up to multiplicative constant depending on $k$, $\beta$, we can lower bound the risk by taking the max of the risks in the two steps:
\begin{align}
\left(\frac1{mh}\wedge 1\{\mathcal{A}\}\right)\vee
\left(
\left(\frac{b(h\wedge 1)}{\log m}\right)^{\beta}
-\frac1{bm^2}
\right)
\label{e_distribute}
\end{align}
whenever $L\ge L_{\beta}:=L_{\beta}^{(1)}\vee L_{\beta}^{(2)}$.
We can use the distributive law to open up the parentheses in \eqref{e_distribute}.
By checking the $h>m^{-\frac1{\beta}}$ and $h\le m^{-\frac1{\beta}}$ cases respectively, it is easy to verify that
\begin{align*}
\frac1{mh}
\vee\left(
\left(\frac{b(h\wedge 1)}{\log m}\right)^{\beta}
-\frac1{bm^2}
\right)
=\Omega\left(\frac{m^{-\frac{\beta}{\beta+1}}}{\log^{\beta}m}\right).
\end{align*}
Next, if $\mathcal{A}$ is true, we evidently have 
\begin{align*}
 1\{\mathcal{A}\}\vee
\left(
\left(\frac{b(h\wedge 1)}{\log m}\right)^{\beta}
-\frac1{bm^2}
\right)
=1=
\Omega\left(\frac{m^{-\frac{\beta}{\beta+1}}}{\log^{\beta}m}\right).
\end{align*}
If $\mathcal{A}$ is not true, then $h> \frac{0.02a}{\log\left(\frac{mq_1}{0.001a}\vee \frac{10}{a}\right)}\wedge 1$, and we have
\begin{align*}
 1\{\mathcal{A}\}\vee
\left(
\left(\frac{b(h\wedge 1)}{\log m}\right)^{\beta}
-\frac1{bm^2}
\right)
&=\left(
\left(\frac{b(h\wedge 1)}{\log m}\right)^{\beta}
-\frac1{bm^2}
\right)
\\
&=\Omega\left(\log^{-2\beta}m\right)
\\
&=
\Omega\left(\frac{m^{-\frac{\beta}{\beta+1}}}{\log^{\beta}m}\right).
\end{align*}
In either case the risk with respect to $L_1$ is $\Omega\left(\frac{m^{-\frac{\beta}{\beta+1}}}{\log^{\beta}m}\right)$. It remains to convert this to a lower bound in $L_2$. 

We consider two cases. First note that by the fast decay condition on the Fourier transform, $k \in \mc{C}^1$. Let $B = B_k$ denote a constant such that
\begin{equation}
\label{eqn:ker_prime_bd}
\sup_{x \in [-1/2, 1/2]} \abs{ k'(x) } \leq B.  
\end{equation}
Set $\Delta = B^{1/2} \vee k(0) \vee 1 $. 

\noindent \textit{Case 1:} $h \leq \Delta$. 

Let $U = \{ |y| \geq \frac{1}{2} + c_{\beta, \Delta, a} \log m \}$, and let $U^c = \mb{R} \backslash U$. If $h \leq \Delta$, then because $X_i \in [-1/2, 1/2]$ and by the exponential decay of $k$, 
\[
\norm{\hat f_X(y) \1_{U} }_1 \leq m^{-2}  
\]
for $c_{\beta, \Delta, a}$ sufficiently large. Thus by Cauchy--Schwarz,
\begin{align*}
\norm{ (\hat f_X - f)\1_{U^c}}_2 &\geq c'_{\beta, \Delta, a} (\log m)^{-1/2} \norm{ (\hat f_X - f)\1_{U^c}}_2 \\
&= c'_{\beta, \Delta, a} (\log m)^{-1/2} \left( \norm{ (\hat f_X - f)}_1 - \norm{ (\hat f_X - f)\1_U}_1  \right) \\
&\geq c'_{\beta, \Delta, a} (\log m)^{-1/2} \left( c_{\beta, k} \left(\frac{m^{-\frac{\beta}{\beta+1}}}{\log^{\beta}m}\right) - m^{-2}  \right) \\
&= \Omega\left(  \frac{m^{-\frac{\beta}{\beta+1}}}{\log^{\beta+ \frac{1}{2}}m} \right)
\end{align*}

\noindent \textit{Case 2:} $h \geq \Delta$

In this case, $k(X_i - y)$ is nearly constant for all $i$. By \eqref{eqn:ker_prime_bd} and Taylor's theorem,
\[
\abs{ k(0) - k\left( \frac{X_i-y}{h} \right)} \leq 2B
\]
for all $y \in [-1/2, 1/2]$ and for all $i$. Hence, for all $y \in [-1/2, 1/2]$, using $h \geq \Delta$,
\[
\hat f_X(y) = \frac{1}{mh} \sum_{i = 1}^m k\left( \frac{X_i-y}{h} \right) \leq \frac{1}{h} ( k(0) + 2B ) \leq 3.
\]
For $L_\beta$ large enough, we see that for the function $f \in \hd(\beta, L_\beta)$ with $f|_{[ 0, \frac{1}{100} ]} \equiv 4$, 
\[
\norm{\hat f_X - f}_2 \geq \norm{(\hat f_X - f)\1_{[ 0, \frac{1}{100} ]}}_1 = \Omega(1). 
\]
\end{proof}

\section{PROOFS FROM SECTION 5}

\subsection{Proof of Theorem \ref{thm:discrepancy_phitai}}

The result is restated below. 

\begin{theorem*}
Let $k_s$ denote the kernel from Section \ref{sec:Cara_results}. The algorithm of \cite{PhiTai19} yields in polynomial time a subset $S$ with $\abs{S} = m = \tilde{O}( n^{\frac{\beta+d}{2\beta + d}} )$ such that the uniformly weighted coreset KDE $\hat g_S$ satisfies 
\[
\sup_{ f \in \mc{P}_{\mc{H}}(\beta, L)  } \E \norm{f - \hat g_S}_2 \leq c_{\beta, d, L}  \, n^{-\frac{\beta}{2\beta + d}} .
\]
\end{theorem*}

\begin{proof}

Here we adapt the results in Section 2 of \citet{PhiTai19} to our setting where the bandwidth $h = n^{-1/(2\beta + d)}$ is shrinking. Using their notation, we define $K_s(x, y) = k_s\left( \frac{x - y}{h} \right)$ and study the kernel discrepancy of the kernel $K_s$. First we verify the assumptions on the kernel (bounded influence, Lipschitz, and positive semidefiniteness) needed to apply their results.

First, the kernel $K_s$ is \textit{bounded influence} \citep[see][Section 2]{PhiTai19} with constant $c_K = 2$ and $\delta = n^{-1}$, which means that
\begin{equation*}
\abs{K_s(x, y)} \leq \frac{1}{n}
\end{equation*} 
if $\abs{x - y}_\infty \geq n^2$. This follows from the fast decay of $\kappa_s$. 

Note that if $x$ and $y$ differ on a single coordinate $i$, then
\[
\abs{ k_s(x) - k_s(y) } \leq \abs{ c (x_i -  y_i) \prod_{j \neq i} \kappa_s(x_j)  } \leq c \abs{x_i - y_i}
\]
because $\abs{\kappa_s(x)} \leq \norm{\psi}_1$ for all $x$ and the function $\kappa_s$ is $c$-Lipschitz for some constant $c$. Hence by the triangle  and Cauchy--Schwarz inequalities, the function $k_s$ is Lipschitz:
\begin{align*}
\abs{ k_s(x) - k_s(y) } \leq d c_k \abs{x - y}_1 \leq d^{3/2} c_\kappa \abs{x - y}_2. 
\end{align*}

Therefore the kernel $K_s(x, y)$ is \textit{Lipschitz} \citep[see][]{PhiTai19} with constant $C_K = d^{3/2} c_\kappa h^{-1}$. Moreover, the kernel $K_s$ is \textit{positive semidefinite} because the Fourier transform of $\kappa_s$ is nonnegative. 

Given the shrinking bandwidth $h = n^{-1/(2\beta + d)}$, we slightly modify the lattice used in \citet[][Lemma 1]{PhiTai19}. Define the lattice
\[
\mc{L} = \left\{ \left. (i_1 \delta, i_2 \delta, \ldots, i_d \delta) \, \right\vert \, i_j \in \mb{Z} \right\},
\]
where
\[
\delta = \frac{1}{c_\kappa d^2 n h^{-1}}. 
\]

The calculation at the top of page 6 of \citet[][Lemma 1]{PhiTai19} yields
\begin{align*}
\disc(X, \chi, y) &:= \abs{ \sum_{i = 1}^n \chi(X_i) K_s( X_i, y)  } 
\\
&\leq \abs{ \sum_{i = 1}^n \chi(X_i) K_s( X_i, y_0)  }  + 1
\end{align*}
where $y_0$ is the closest point to $y$ in the lattice $\mc{L}$, and $\chi$ assigns either $+1$ or $-1$ to each element of $X= X_1, \ldots, X_n$. Moreover, with the bounded influence of $K_s$, if 
\[
\min_{i} \abs{y - X_i}_\infty \geq n^{2},
\] 
then 
\[
\disc(X, \chi, y) = \abs{ \sum_{i = 1}^n \chi(X_i) K_s( X_i, y)  } \leq 1.
\]
On defining
\[
\mc{L}_X = \mc{L} \cap \{ y: \, \min_i \abs{y - X_i}_\infty \leq n^2 \},
\]
we see that 
\begin{equation*}
\max_{y \in \mb{R}^d} \disc(X, \chi, y) \leq \max_{y \in \mc{L}_X} \disc(X, \chi, y) + 1
\end{equation*}
for all signings $\chi: X \to \{-1, +1\}$. This is precisely the conclusion of \citet[][Lemma 1]{PhiTai19}.

This established, the positive definiteness and bounded diagonal entries of $K_s$ and \citet[][Lemmas 2 and 3]{PhiTai19} imply that 
\[
\disc_{K_s} = O(\sqrt{d \log n}). 
\]

Given $\eps > 0$, the halving algorithm can be applied to $K_s$ as in \citet[][Corollary 5]{PhiTai19} to yield a coreset $X_S$ of size $m = O( \eps^{-1} \sqrt{ d \log \eps^{-1} })$ such that
\[
\norm{ \frac{1}{n} \sum_{j = 1}^n K_s(X_j, y) - \frac{1}{m} \sum_{j \in S} K_s(X_j, y) }_\infty \leq \eps. 
\]
Rescaling by $h^{-d}$, we have
\[
\norm{\hat f  - \unif }_\infty = \norm{ \frac{1}{n} \sum_{j = 1}^n k_s(X_j, y) - \frac{1}{m} \sum_{j \in S} k_s(X_j, y) }_\infty \leq \eps h^{-d}. 
\]

Recall from Section \ref{sec:thm2_pf} that $\hat f$ attains the minimax rate of estimation on $\hd(\beta, L)$. Thus setting $\eps = h^d n^{-\beta/(2\beta + d)}$ we get a coreset of size $\tilde{O}_d( n^{\frac{\beta + d}{2\beta + d}})$ that attains the minimax rate $c_{\beta, d, L} \, n^{-\beta/(2\beta + d)}$, as desired. Moreover, by the results of \citet{PhiTai19}, this coreset can be constructed in polynomial time. 

\end{proof}

\medskip 
 
\noindent \textbf{Acknowledgments} We thank Cole Franks for helpful discussions regarding algorithmic aspects of Carath\'{e}odory's theorem. 
    
	\bibliographystyle{plainnat}
    \bibliography{Coresets_arxiv_bib}
	
	\end{document}